\documentclass[reqno]{amsart}
\usepackage{amsmath,amssymb,latexsym,url,graphicx,tikz,times,mathptmx}

\newtheorem{theorem}{Theorem}[section]
\newtheorem{conjecture}[theorem]{Conjecture}

\newtheorem{definition}[theorem]{Definition}

\newtheorem{proposition}[theorem]{Proposition}

\def\ConjectureName{Edge Perturbation Conjecture}
\begin{document}

\title[Fourier Levels and Almost Sure Bounds]{Fourier Levels and Almost Sure Bounds on Higher-Order Derivatives in First-Passage Percolation} 

\author[I. Mati\'c]{
Ivan Mati\'c}
\address{
Baruch College, City University of New York}\email{ivan.matic@baruch.cuny.edu}
\author[R. Radoi\v ci\'c]{
Rado\v s Radoi\v ci\'c} \address{
Baruch College, City University of New York}\email{rados.radoicic@baruch.cuny.edu} 
\author[D. Stefanica]{
Dan Stefanica}\address{
Baruch College, City University of New York}\email{dan.stefanica@baruch.cuny.edu}

\maketitle

\begin{abstract}
The variance of first-passage percolation admits a decomposition into Fourier levels indexed by the order of environment derivatives. These Fourier levels capture how local perturbations of different orders contribute to global fluctuations. In this paper, we investigate higher-order environment derivatives and their Fourier-level structure. We prove that derivatives of orders \(k\in\{2,3,4\}\) are almost surely bounded below by \(-\binom{k-2}{\lceil\frac{k-2}2\rceil}\) and above by \(\binom{k-2}{\lceil\frac{k-2}2\rceil}\). We conjecture that these are the correct bounds for all \(k\), and we construct explicit environments showing that these extreme values can indeed be attained.
\end{abstract}


\section{Introduction}

\subsection{Model}

We use the same model and definitions as in \cite{FirstPaper2025}. We refer the reader to that paper for a more comprehensive historical discussion and overview of the literature. For completeness, we briefly recall the model and state the results needed for our proofs.

Let \( a < b \) be two fixed positive real numbers, and let \( p \in (0,1) \) be fixed. We consider the finite subgraph of the integer lattice \( \mathbb{Z}^d \) (with \( d \geq 2 \)) induced by the box \( [-2n,2n]^d \). Two vertices \( (x_1,\dots,x_d) \) and \( (y_1,\dots,y_d) \) are connected by an edge if and only if they are nearest neighbors, that is,
\[
|x_1-y_1| + \cdots + |x_d-y_d| = 1.
\]

Each edge \( e \) is independently assigned a random passage time, equal to \( a \) with probability \( p \) and to \( b \) with probability \( 1-p \). Let \( W_n \) denote the set of all edges in this graph, and define the sample space \( \Omega_n = \{a,b\}^{W_n} \), whose elements \( \omega \in \Omega_n \) represent configurations of passage times.

Given a configuration \( \omega \) and a path \( \gamma \) (a sequence of adjacent edges), the passage time \( T(\gamma,\omega) \) is defined as the sum of the passage times of the edges in \( \gamma \). For vertices \( u,v \), we write \( f(u,v,\omega) \) for the minimal passage time among all paths connecting \( u \) to \( v \) in the environment \( \omega \).

When the destination vertex is fixed, we write \( f_n(\omega) \), or simply \( f_n \), for \( f(0,nv,\omega) \).

A path \( \gamma \) is called a \emph{geodesic} if the minimum \( f_n(\omega) \) is attained at \( \gamma \), that is, if \( f_n(\omega) = T(\gamma,\omega) \).

\subsection{The torus model}

For some of our results, we work with a simplified and more symmetric model introduced in \cite{benjamini2006}, in which first-passage percolation is considered on the discrete torus \( \mathbb{Z}_n^d \). We use the superscript \( \tau \) to indicate this setting, writing \( f_n^{\tau}(\omega) \), \( f_n^{\tau} \), or simply \( f^{\tau} \) when the dependence on \( n \) is clear.

The \( d \)-dimensional discrete torus is the graph with vertex set \( \mathbb{Z}_n^d \), in which two vertices \( u \) and \( v \) are connected by an edge if and only if they differ by \( \pm 1 \) modulo \( n \) in exactly one coordinate.

Let \( \Gamma \) denote the collection of all nearest-neighbor paths on \( \mathbb{Z}_n^d \) that wrap once around the torus in the \( x_1 \)-direction. More precisely, \( \gamma=(v_0,\dots,v_m)\in\Gamma \) if \( v_0 \) and \( v_m \) have identical coordinates except possibly in the first coordinate, with \( (v_0)_1=0 \) and \( (v_m)_1=n-1 \).

We define \( f_n^{\tau} \) to be the minimal passage time among all such paths,
\begin{eqnarray*}
f_n^{\tau}(\omega)
=
\min\left\{T(\gamma,\omega): \gamma \in \Gamma \right\}.
\end{eqnarray*}

\subsection{Environment derivatives}
If we denote by $W_n$ the set of all edges, then the sample space is $\Omega_n=\{a,b\}^{W_n}$. We will often omit the subscript $n$, when there is no danger of confusion. For each edge $j$ and each $\omega\in \Omega$, we define $\sigma_j^a(\omega)$ as the element of $\Omega$ whose $j$-th coordinate is changed from $\omega_j$ to $a$, regardless of what the original value $\omega_j$ was. The operation $\sigma_j^b$ is defined in analogous way. Formally, for $\delta\in\{a,b\}$, we define $\sigma_j^{\delta}:\Omega\to\Omega$ with 
\begin{eqnarray*}
\left[\sigma_j^{\delta}(\omega)\right]_k&=&\left\{\begin{array}{ll}\omega_k,& k\neq j,\\
\delta, &k=j.\end{array}\right.  
\end{eqnarray*}
If $\varphi:\Omega\to\mathbb R$ is any random variable, then the {\em first order environment derivative} $\partial_j\varphi$ is the random variable defined as 
\begin{eqnarray*}
\partial_j\varphi&=&\varphi\circ \sigma_j^b-\varphi\circ \sigma_j^a.  
\end{eqnarray*}

For two distinct vertices $k$ and $l$, we will give the name {\em second order environment derivative} to the quantity $\partial_{k}\partial_{l}\varphi$. In general, if $S$ is a non-empty subset of $W$, the operator $\partial_S\varphi$ is defined recursively as 
\begin{eqnarray}
\partial_S\varphi&=&\partial_{S\setminus\{j\}}\left(\partial_j\varphi\right), \label{eqn:definitionPartialS}
\end{eqnarray}
where $j$ is an arbitrary element of $S$. The definition \eqref{eqn:definitionPartialS} is independent on the choice of $j$, since a simple induction can be used to prove that for $S=\{s_1,\dots, s_m\}$, the following holds 
\begin{eqnarray}
\partial_S\varphi&=&\sum_{\theta_1\in\{a,b\}} \cdots\sum_{\theta_m\in\{a,b\}} (-1)^{\mathbf1_a(\theta_1)+ \cdots+\mathbf1_a(\theta_m)}\varphi\circ\sigma_{s_1}^{\theta_1} \circ\cdots\circ \sigma_{s_m}^{\theta_m}. \label{eqn|LeibnitzRule}
\end{eqnarray}
The function $\mathbf1_a:\{a,b\}\to \{0,1\}$ in \eqref{eqn|LeibnitzRule} assigns the value $1$ to $a$ and $0$ to $b$. 

\subsection{Variance decomposition and Fourier levels} 

A thorough understanding of environment derivatives would lead to a complete understanding of the variance. As shown in \cite{FirstPaper2025}, the variance admits the decomposition
\begin{eqnarray}
\mathrm{var}(f)
&=&
\sum_{M\subseteq W,\, M\neq\emptyset}
\left(p(1-p)\right)^{|M|}
\left(\mathbb{E}\left[\partial_M f\right]\right)^2,
\label{eqn|varianceFormula}
\end{eqnarray}
which highlights the central role of the $L^2$ norms of the derivatives $\partial_M f$ in determining fluctuation behavior. Despite their importance, general bounds for these quantities are not yet available.

We rewrite \eqref{eqn|varianceFormula} as
\begin{eqnarray}
\mathrm{var}(f)
=
\sum_{k=1}^{\infty} (p(1-p))^k\, \Sigma_k(f),
\label{eqn|varianceFormulaWeights}
\end{eqnarray}
where $\Sigma_k(f)$ is the level-$k$ Fourier sum of $f$, defined by
\begin{eqnarray*}
\Sigma_k(f)
=
\sum_{\substack{M\subseteq W\\ |M|=k}}
\left(\mathbb{E}\big[\partial_M f\big]\right)^2. 
\end{eqnarray*}

In this paper, we prove the following theorem, which provides bounds on the first and second Fourier levels in the torus model.

\begin{theorem}\label{thm|BoundFourierLevel2Torus}
If $a$ is an integer multiple of $b-a$, then there exists a constant $C$, depending only on $a$, $b$, $p$, and $d$, such that the level-$1$ and level-$2$ Fourier sums of $f^{\tau}$ satisfy
\begin{eqnarray}
\Sigma_1(f^{\tau}) &\leq& C n^{2-d},
\label{eqn|BoundFourierLevel1Torus}\\
\Sigma_2(f^{\tau}) &\leq& C n^{3-d}.
\label{eqn|BoundFourierLevel2Torus}
\end{eqnarray}
\end{theorem}

The variance decomposition \eqref{eqn|varianceFormula} is 
essentially Parseval's identity for the Walsh--Fourier expansion 
of $f$; the Fourier coefficients can be expressed in terms of 
environment derivatives via a discrete integration by parts, as 
noted in \cite{Tal2017} and made explicit in \cite{FirstPaper2025}. 
The companion paper \cite{FirstPaper2025} further establishes a 
generalization of Talagrand's and Tanguy's inequalities 
\cite{talagrand1994,Tanguy2018}.

\subsection{What the Fourier-level bounds suggest}

Theorem~\ref{thm|BoundFourierLevel2Torus} casts the variance 
conjecture in a new light, not by bringing it closer to resolution, 
but by making it more mysterious. In dimension $d=3$, the first two 
Fourier levels contribute at most $O(1)$ to the variance. If the 
widely conjectured bound $\mathrm{var}(f)=O(n^{2\chi})$ with 
$\chi>0$ holds in dimension $3$, then the divergence cannot come 
from the first two levels. It must come entirely from higher Fourier 
levels. This would be a striking phenomenon: the contributions of 
individual edges and pairs of edges to the variance are negligible, 
yet the cumulative effect of higher-order interactions produces 
polynomial growth. On the other hand, if higher Fourier levels are 
also well-controlled, as we conjecture, particularly in dimensions 
$d\geq 3$ where the $L^2$ norms $\|\partial_M f\|_2$ may decay 
exponentially in $k$, then the variance itself would be $O(1)$, 
implying $\chi=0$ in dimension $3$. Either conclusion would be 
remarkable, and both remain out of reach.

\subsection{Almost sure bounds on environment derivatives}
The sequence $(\mathcal U_1$, $\mathcal U_2$, $\dots)$ represents the optimal upper bounds for environment derivatives. The number $\mathcal U_k$ is defined as the best upper bound on the $k$-th order environment derivative, i.e. 
\begin{eqnarray}
\mathcal U_k&=&\frac1{b-a} \max\left\{\partial_Sf_n(\omega):n\in\mathbb N, S\subseteq W_n, |S|=k, \omega\in \Omega_n\right\}. \label{eqn|DefinitionUn}
\end{eqnarray}
We define an analogous sequence $(\mathcal L_1, \mathcal L_2, \dots)$ 
of optimal lower bounds. For each $k\in\mathbb N$, 
\begin{eqnarray}
\mathcal L_k&=&\frac1{b-a} \min\left\{\partial_Sf_n(\omega):n\in\mathbb N, S\subseteq W_n, |S|=k, \omega\in \Omega_n\right\}. \label{eqn|DefinitionLn}
\end{eqnarray}
 \begin{theorem}\label{thm|firstFew} 
In dimensions $d \geq 3$, the first four values of $(\mathcal{U}_k)$ and $(\mathcal{L}_k)$ are shown in the table below.
\begin{eqnarray}&&
\begin{array}{|c|r|r|r|r|}
\hline 
k&{\quad}1&{\quad}2&{\quad}3&{\quad}4\\
\hline
\mathcal U_k&1&1&1&2\\
\hline 
\mathcal L_k&0&-1&-1&-2\\
\hline
\end{array}
\label{eqn|firstFew}
\end{eqnarray}
\end{theorem}
\begin{theorem}\label{thm|generalUpperBounds}
The sequences $(\mathcal U_k)$ and $(\mathcal L_k)$ satisfy 
\begin{eqnarray} 
&& \mathcal U_{k+1}\leq \mathcal U_k -\mathcal L_k\quad\mbox{and}\quad \mathcal L_{k+1}\geq \mathcal L_k-\mathcal U_k, \label{eqn|gBoundFibonacci}\end{eqnarray} for all $k\geq 1$. Moreover, for all $k\geq 2$,
\begin{eqnarray}
&& \mathcal U_k \leq 2^{k-2}\quad\mbox{and}\quad \left|\mathcal L_k\right| \leq 2^{k-2}. \label{eqn|uBoundsUL}\end{eqnarray}
Also, for all integers $k \geq 2$, the following holds:
\begin{eqnarray}&& \mathcal U_k \geq \binom{k-2}{\lceil\frac{k-2}2\rceil} \quad\mbox{and}\quad \left|\mathcal L_k\right| \geq \binom{k-2}{\lceil\frac{k-2}2\rceil}.\label{eqn|lBoundsUL}
\end{eqnarray}
\end{theorem}
Theorem \ref{thm|generalUpperBounds} implies that $(\mathcal U_k)$ and $(\mathcal L_k)$ grow exponentially in $k$. Due to Stirling's formula, these sequences are between $\frac{2^{k-2}}{k-2}$ and $2^{k-2}$.

We conjecture that \eqref{eqn|lBoundsUL} are equalities. However, we have formal proofs only for $k\leq 4$ and we believe 
that it is possible to obtain at least a computer-assisted proof for $k=5$. We don't have the proof yet.

\subsection{The {\ConjectureName}}
\begin{conjecture}[\ConjectureName]
For all integers $k\geq2$, the following equalities hold 
\begin{eqnarray}&& \mathcal U_k=\binom{k-2}{\lceil\frac{k-2}2\rceil}\quad\mbox{and}\quad 
\mathcal L_k=-\binom{k-2}{\lceil\frac{k-2}2\rceil}. \label{eqn|ULEquality}
\end{eqnarray}
\end{conjecture}
The inequalities \eqref{eqn|lBoundsUL} provide a lower bound of $\mathcal U_k$ and an upper bound for $\mathcal L_k$. The conjectured equalities hold for $k\in\{2,3,4\}$. Moreover, the following easy combinatorial identity provides a really miraculous jump from odd $k$ to even $k$. 

Namely, if $k\geq 3$ is an odd integer, then 
\begin{eqnarray}
\binom{k-2}{\lceil\frac{k-2}2\rceil}+\binom{k-2}{\lceil\frac{k-2}2\rceil}&=& \binom{k-1}{\lceil\frac{k-1}2\rceil}.\label{eqn|MiraculousIdentity}
\end{eqnarray} 

The equality \eqref{eqn|MiraculousIdentity} makes the conjecture even more believable: if \eqref{eqn|ULEquality} holds for an odd $k$, then it would immediately hold for $k+1$ due to the Fibonacci-type bounds 
\eqref{eqn|gBoundFibonacci} and the established bounds \eqref{eqn|lBoundsUL}.

The bounds \eqref{eqn|lBoundsUL} were obtained by constructing specific configurations on which the derivatives attain these extreme values. Of course, there is no guarantee that some other configurations won't violate the {\ConjectureName}. In our research, we considered several possible families of configurations and wrote computer programs that would go over the members of the families and search for extreme derivatives. The best that we were able to find are the families of configurations that we call {\em disjoint lanes}, which are surprisingly simple. We have considered variations in which the numbers of lanes was higher than 2, but these turned out to be sub-optimal. 

The extremal configurations require $n$ to be large, making exhaustive computational verification infeasible.

\subsection{Related work and context}

First-passage percolation was introduced by Hammersley and Welsh
\cite{HammersleyWelsh1965}; we refer to \cite{AuffingerDamronHanson2017}
for a comprehensive survey. The best known upper bound
$\mathrm{var}(f_n)\leq C\,n/\log n$ is due to Benjamini, Kalai, and
Schramm \cite{benjamini2006}, extended to general edge-weight
distributions in \cite{benaim2008} and \cite{DamronHansonSosoe2015}.
The widely conjectured sharp bound is $\mathrm{var}(f_n)=O(n^{2\chi})$
with $\chi=1/3$ in two dimensions \cite{AuffingerDamronHanson2017};
the rigorous upper bound $\chi\leq 1/2$ is due to Kesten
\cite{Kesten1993}, and a proof that $\chi<1/2$ remains out of reach.
Related variance inequalities in the Boolean function setting appear
in \cite{Przybylowski2024}, the superconcentration framework is
developed in \cite{chatterjee2014}, and higher-order concentration
inequalities for functions of independent random variables are studied
in \cite{BobkovGotzeSambale2019}; the finite-difference calculus
underlying environment derivatives is treated in \cite{ODonnell2014}.

The proof of Theorem~\ref{thm|BoundFourierLevel2Torus} requires
understanding when second-order derivatives are nonzero, which in turn
requires analyzing configurations of geodesics. In the continuously
distributed setting, geodesics are unique and influential edges coincide
with essential ones; polynomial bounds $\mathbb{P}(A_j)\leq Cn^{-\xi}$
were established in \cite{DamronHanson2017}, strengthened by removing
differentiability assumptions in \cite{AhlbergHoffman2019}, and
culminated in the coalescence results of \cite{DembinElboimPeled2024}.
In the two-valued setting studied here, multiple geodesics can coexist,
introducing a richer structure of essential, semi-essential, influential,
and very influential edges developed in \cite{FirstPaper2025}; the
present paper builds directly on this structure in
Section~\ref{sec|HigherOrderInfluential}. Geodesic coalescence in FPP
has been studied in \cite{Hoffman2008}, \cite{Alexander2023},
\cite{Seppalainen2020}, and \cite{KrishnanRassoulAghaSeppalainen2023}.
The fluctuation and transversal exponents satisfy the KPZ scaling
relation $\chi=2\xi-1$, proved in \cite{NewmanPiza1995,chatterjee2013}
and generalized in \cite{AuffingerDamron2014}; the asymptotic shape of
FPP balls is studied in \cite{CoxDurrett1981} and
\cite{ChatterjeeDey2016}.

First-passage percolation belongs to the Kardar--Parisi--Zhang
universality class \cite{AlbertsKhaninQuastel2014,CorwinGhosalHammond2021}
and can be viewed as the zero-temperature limit of directed polymers
\cite{Zygouras2024}. The value $\chi=1/3$ in two dimensions was
established rigorously in the TASEP model by Johansson
\cite{Johansson2000}, with Tracy--Widom fluctuations
\cite{TracyWidom1994}. The models we study here are discrete, but first-passage percolation has been successfully generalized to Euclidean spatial models \cite{HowardNewman1997}. Large deviation estimates and scaling relations for spatial and lattice FPP models appear in \cite{BasuGangulySly2021} and \cite{BasuSidoraviciusSly2023}. In the continuous PDE setting of
random Hamilton--Jacobi equations, the $n/\log n$ variance bound was
obtained in \cite{MaticNolen2012}, homogenization results appear in
\cite{RezakhanlouTarver2000,ArmstrongCardaliaguetSouganidis2014,
DaviniKosyginaYilmaz2023}, and differentiability of the limit shape
is studied in \cite{BakhtinDow2024}.

The almost-sure bounds proved in this paper and the \ConjectureName{}
connect to several threads in combinatorics. Glasby and Paseman
\cite{GlasbyPaseman2022,GlasbyPaseman2024} studied extrema of weighted
binomial sums arising in coding theory; Byun and Poznanovi\'{c}
\cite{ByunPoznanovic2024} extended this to a parametric family. The
alternating binomial sums appearing in the proof of
Theorem~\ref{thm|dLanes} connect to lattice path enumeration via the
Lindstr\"{o}m--Gessel--Viennot method
\cite{Fulmek2012,Krattenthaler2010,Lee2022}, to explicit formulae for
generalized binomial coefficients \cite{Petrov2016}, and to
alternating-sum identities \cite{GuoJouhetZeng2007,ElBachraoui2021}.
   


\section{Essential and influential edges}
Except for Proposition \ref{thm|Straightforward} below, the results in this section were proved in \cite{FirstPaper2025}. Proposition \ref{thm|Straightforward} is trivial, 
but so important that it must be listed. 
\begin{proposition}\label{thm|Straightforward} For every $i\neq j$, every $\alpha,\beta\in\{a,b\}$, and every random variable $\varphi$, 
\begin{eqnarray} 
\sigma_i^{\alpha}\circ \sigma_i^{\beta}&=&\sigma_i^{\alpha};   \label{eqn|propertyA}\\
\sigma_i^{\alpha}\circ \sigma_j^{\beta}&=&\sigma_j^{\beta}\circ\sigma_i^{\alpha};  \label{eqn|propertyB} \\
(\partial_i \varphi)\circ \sigma_i^{\alpha} &=& \partial_i \varphi;  \label{eqn|propertyC} \\
\partial_i\partial_i \varphi&=&0;  \label{eqn|propertyD} \\
\partial_i\partial_j \varphi&=&\partial_j\partial_i\varphi;  \label{eqn|propertyE} \\
\varphi\cdot 1_{\omega_i=\alpha}&=&\varphi\circ \sigma_i^{\alpha}\cdot 1_{\omega_i=\alpha}.  \label{eqn|propertyF}
\end{eqnarray}
\end{proposition}

For a fixed edge $j\in W_n$, define the events $A_j$ and $\hat A_j$ with 
\begin{eqnarray}
A_j&=&\left\{\omega\in\Omega_n:\partial_jf(\omega)\neq 0\right\}; \label{eqn|ImportantCellsIndicator}\\
\hat A_j&=&\left\{\omega\in\Omega_n:\partial_j f(\omega)=b-a\right\}. \label{eqn|HatImportantCellsIndicator}
\end{eqnarray}
The edge $j$ is called {\em influential} if the event $A_j$ occurred, and {\em very influential} if $\hat A_j$ occurred.
The following identities are set-theoretic consequences of 
\eqref{eqn|ImportantCellsIndicator}, 
\eqref{eqn|HatImportantCellsIndicator}, and \eqref{eqn|propertyC}; 
a formal derivation appears in \cite{FirstPaper2025}, and a 
generalization is proved in \eqref{eqn|invarianceLevelSetsGen} below. 
For every edge $j$ and every $\xi\in\{a,b\}$,
\begin{eqnarray}
\left(\sigma_j^{\xi}\right)^{-1}(A_j)=A_j 
&\mbox{and}& 
\left(\sigma_j^{\xi}\right)^{-1}(\hat A_j)=\hat A_j. 
\label{eqn|invarianceLevelSets}
\end{eqnarray}

The edge $j$ is called {\em essential} if each geodesic passes through $j$, and {\em semi-essential} if at least one geodesic passes through $j$. We will denote by $E_j$ the event that the edge $j$ is essential and by $\hat E_j$ the event that the edge $j$ is semi-essential. 

\begin{proposition}\label{thm|AiDefinition} The events $E_j$, $A_j$, $\hat E_j$, and $\hat A_j$ satisfy
\begin{eqnarray}
 A_j&=&(\sigma_j^a)^{-1}\left(E_j\right); \label{eqn|AiDefinition}\\
 \hat A_j&=&(\sigma_j^b)^{-1}(\hat E_j). \label{eqn|HatAiDefinition}
\end{eqnarray}
\end{proposition}

\begin{proposition}\label{thm|forComputer} The following two propositions hold for every $\omega\in\Omega$.
\begin{enumerate}
\item[(a)] If $\sigma_j^a(\omega)\in E_j^C$, then $f(\sigma_j^b(\omega))=f(\sigma_j^a(\omega))$;
\item[(b)] If $\sigma_j^b(\omega)\in \hat E_j$, then $f(\sigma_j^b(\omega))=f(\sigma_j^a(\omega))+(b-a)$.
\end{enumerate}
\end{proposition}

The sets $\{\omega_j=a\}$ and $\{\omega_j=b\}$ are the ranges of the transformations $\sigma_j^a$ and $\sigma_j^b$, i.e. \begin{eqnarray}
\sigma_j^a\left(\Omega\right)=\{\omega_j=a\}\quad\mbox{ and }\quad 
\sigma_j^b\left(\Omega\right)=\{\omega_j=b\}.
\label{eqn|fixedPointsSigmaJAB}
\end{eqnarray}

\begin{proposition} For every $j\in W$, the events $E_j$, $\hat E_j$, $A_j$, and $\hat A_j$ satisfy 
\begin{eqnarray}
&&\sigma_j^a(A_j)=\sigma_j^a(A_j\cap\{\omega_j=b\}) 
=A_j\cap \left\{\omega_j=a\right\} 
=E_j\cap \left\{\omega_j=a\right\}; \label{eqn:AEinclusion}\\
&&\sigma_j^b(\hat A_j) = \sigma_j^b(\hat A_j\cap\{\omega_j=a\})  
=\hat A_j\cap \left\{\omega_j=b\right\} 
=\hat E_j\cap \left\{\omega_j=b\right\}. \label{eqn:HatAEinclusion}
\end{eqnarray}

\end{proposition}

\begin{proposition} For every $j\in W$, the events $E_j$, $\hat E_j$, $A_j$, and $\hat A_j$ satisfy
\begin{eqnarray}
&&E_j\subseteq \hat E_j,\quad \hat A_j\subseteq A_j,
\nonumber\\
&&E_j\subseteq A_j ,\quad \hat A_j\subseteq \hat E_j.
\label{eqn:EAinclusion}
\end{eqnarray}
\end{proposition}

\begin{proposition} \label{thm|MonotonicityOfGeodesic} Assume that $\omega\in E_j$. A path $\gamma$ is a geodesic on $\omega$ if and only if it is a geodesic on $\sigma_j^a(\omega)$. 
\end{proposition}

If $\overrightarrow \alpha\in \{a,b\}^m$ and $\overrightarrow v\in W^m$, define $\sigma_{\overrightarrow v}^{\overrightarrow\alpha}:\Omega\to \Omega$ as 
\begin{eqnarray}
\sigma_{\overrightarrow v}^{\overrightarrow\alpha}&=&\sigma_{v_1}^{\alpha_1}\circ\cdots\circ\sigma_{v_m}^{\alpha_m},
\label{eqn|MultiDimensionSigma}
\end{eqnarray}
where $\alpha_1$, $\dots$, $\alpha_m$ are the components of $\overrightarrow\alpha$ and $v_1$, $\dots$, $v_m$ are the components of $\overrightarrow v$.



\section{Disjoint lanes} \label{sec|Dls}
Our goal is to construct special, extreme environments in which the derivatives attain very large positive values and very small negative values. These environments will be used to establish the bounds in \eqref{eqn|lBoundsUL} from Theorem \ref{thm|generalUpperBounds}.  

The proofs involve somewhat lengthy algebraic calculations. Such calculations are omitted here and are presented in the Appendix.

In this section we will consider the first-passage percolation between two vertices, that we will call source and sink. With minimal modifications, the arguments apply to the first-passage percolation time on torus model from 
\cite{benjamini2006}.

Fix non-negative integers $m_1$, $m_2$, $\beta_1$, and $\beta_2$ for which $m_1+m_2\geq 2$; then fix a huge integer $N$ that satisfies
\[N>2^{100\cdot (1+a+b+\frac1{b-a}+m_1+m_2+\beta_1+\beta_2)},\] and, finally, choose an even bigger $n$ such that $n>100N\cdot \lceil\frac1{b-a}\rceil$. We will construct our example on $[-2n,2n]^d$. The source will be the vertex $O=(-2n,0,\dots, 0)$ and the sink will be the vertex $V=(2n,0,\dots, 0)$. 
Let us identify the points $O'=(-2n+N,0,\dots, 0)$, $V'=(2n-N,0,\dots, 0)$, $O_1=(-2n+N, N,0,\dots, 0)$, $O_2=(-2n+N,-N,0,\dots, 0)$, $V_1=(2n-N,N,0,\dots, 0)$, $V_2=(2n-N,
-N,0,\dots, 0)$, and then build the paths $\gamma_1$ and $\gamma_2$ from source to sink in the following way: 
$\gamma_1$ is the shortest path consisting of $2n+2N$ vertices that contains $O$, $O'$, $O_1$, $V_1$, $V'$, and $V$; $\gamma_2$ is the shortest path consisting of $2n+2N$ vertices that contains $O$, $O'$, $O_2$, $V_2$, $V'$, and $V$. The configuration is shown in Figure \ref{fig|disjointLanes01}.

\begin{figure}[t]
\centering
\begin{tikzpicture} 
\draw[thick,dashed] (0,0) -- (1,0);
\draw[thick,dashed] (9,0) -- (10,0); 
\draw[thick,dashed] (1,0) -- (1,-1)--  (9,-1) -- (9,0);  
\draw[thick,dashed] (1,0) -- (1,1) --  (9,1) -- (9,0); 

\node at (0,0) {$\bullet$}; 
\node[left] at (0,0) {$O$};
\node at (1,0) {$\bullet$}; 
\node[above left] at (1,0) {$O'$};
\node at (1,1) {$\bullet$}; 
\node[above left] at (1,1) {$O_1$};
\node at (1,-1) {$\bullet$}; 
\node[below left] at (1,-1) {$O_2$};
\node at (10,0) {$\bullet$}; 
\node[right] at (10,0) {$V$}; 
\node at (9,0) {$\bullet$}; 
\node[above right] at (9,0) {$V'$};
\node at (9,1) {$\bullet$}; 
\node[above right] at (9,1) {$V_1$};
\node at (9,-1) {$\bullet$}; 
\node[below right] at (9,-1) {$V_2$};
\node[below] at (5.1,-1) {$\gamma_2$};
\node[above] at (5.1,1) {$\gamma_1$};
\end{tikzpicture}
\caption{Paths $\gamma_1$ and $\gamma_2$.}
\label{fig|disjointLanes01}
\end{figure}

On the section $O_1V_1$ of the path $\gamma_1$, starting at the vertex $(-2n+3N, N,0,\dots, 0)$, which is $2N$ units to the right of $O_1$, choose a block of $m_1$ consecutive edges and call it $C_1$. Similarly, $C_2$ is the block of $m_2$ consecutive edges on $\gamma_2$ that starts at the point $(-2n+3N,-N,0,\dots, 0)$ which is $2N$ units to the right of $O_2$. 

Let $B_1$ be the block of $\beta_1$ consecutive edges on $\gamma_1$ that terminates at the edge $(2n-3N,N,0,\dots, 0)$, which is $2N$ units to the left of $V_1$. Let $B_2$ be the block of $\beta_2$ consecutive edges on $\gamma_2$ that terminates at the edge $(2n-3N,N,0,\dots, 0)$, which is $2N$ units to the left of $V_2$. The sets $C_1$, $C_2$, $B_1$, and $B_2$ are shown in Figure \ref{fig|disjointLanes}.

\begin{figure}[t]
\centering
\begin{tikzpicture} 
\draw[thick,dashed] (0,0) -- (1,0);
\draw[thick,dashed] (9,0) -- (10,0); 
\draw[thick,dashed] (1,0) -- (1,-1)--  (2.5,-1); 
\draw[thick,dashed] (7.5,-1) -- (9,-1) -- (9,0);
\draw[thick,dashed] (1,0) -- (1,1) -- (2.5,1);
\draw[thick,dashed] (7.5,1) -- (9,1) -- (9,0);
\draw[thick,dashed] (4.1,1) -- (4.6,1);
\draw[thick,dashed] (5.7,1) -- (6.2,1);
\draw[thick,dashed] (5.7,-1) -- (6.2,-1);
\draw[thick,dashed] (4.1,-1) -- (4.6,-1);
\draw[very thick] (6.2,1)--(7.5,1);
\draw[very thick] (6.2,-1)--(7.5,-1);
\draw[very thick] (2.5,0.9) -- (4.1,0.9) -- (4.1,1.1) -- (2.5,1.1) -- cycle;    
\draw[very thick] (2.5,-0.9) -- (4.1,-0.9) -- (4.1,-1.1) -- (2.5,-1.1) -- cycle;  

\draw[very thick] (7.5,0.9) -- (6.2,0.9) -- (6.2,1.1) -- (7.5,1.1) -- cycle;    
\draw[very thick] (7.5,-0.9) -- (6.2,-0.9) -- (6.2,-1.1) -- (7.5,-1.1) -- cycle;    

\node[above] at (1.5,1.0) {$\gamma_1$};
\node[below] at (1.5,-1.0) {$\gamma_2$};
\node[above] at (8.5,1.0) {$\gamma_1$};
\node[below] at (8.5,-1.0) {$\gamma_2$};
\node[above] at (3.3,1.1) {$C_1$}; 
\node[below] at (3.3,-1.1) {$C_2$}; 
\node[above] at (6.9,1.1) {$B_1$};
\node[below] at (6.9,-1.1) {$B_2$};
\node at (5.1,-1) {$\cdots$};
\node at (5.1,1) {$\cdots$};
\end{tikzpicture}
\caption{Position of sets $C_1$, $C_2$, $B_1$, and $B_2$ on the paths $\gamma_1$ and $\gamma_2$.}
\label{fig|disjointLanes}
\end{figure}

The set $S$ will be defined to be $S=C_1\cup C_2$. We will now make an assignment $\omega$ of numbers from $\{a,b\}$ to every edge in such a way that $\partial_Sf(\omega)$ is extreme. 

The environment $\omega$ assigns the value $b$ to all the edges from $B_1$ and $B_2$ and all the edges outside of $\gamma_1\cup \gamma_2$. The environment $\omega$ assigns the value $a$ to every edge in $\gamma_1\cap \gamma_2\setminus (B_1\cup B_2)$. 
 
Define $D(m_1,m_2;\beta_1,\beta_2)$ to be the environment derivative $\partial_Sf(\omega)$ for the pair $(S,\omega)$ that was described above.

\begin{theorem}
\label{thm|dLanes} The number $D(m_1,m_2;\beta_1,\beta_2)$ is $0$ if $\beta_1-\beta_2\geq m_2$ or $\beta_2-\beta_1\geq m_1$. 
If both of the inequalities $\beta_1-\beta_2\leq m_2-1$ and $\beta_2-\beta_1\leq m_1-1$ are satisfied, then 
\begin{eqnarray}
D(m_1,m_2;\beta_1,\beta_2)&=& (b-a)\cdot (-1)^{m_1+m_2+\beta_1+\beta_2}\cdot\binom{m_1+m_2-2}{m_1+\beta_1-\beta_2-1}.
\label{eqn|dLanes}
\end{eqnarray}
\end{theorem}

\begin{proof} The proof is in the Appendix. \end{proof}

The next proposition implies the bounds \eqref{eqn|lBoundsUL} in Theorem \ref{thm|generalUpperBounds}.

\begin{proposition}\label{thm|badSets} For every $m\geq 2$, the following hold
\begin{eqnarray}&& \mathcal U_m\geq \binom{m-2}{\lceil\frac{m-2}2\rceil}\quad\mbox{and}\quad 
\mathcal L_m\leq -\binom{m-2}{\lceil\frac{m-2}2\rceil}. \label{eqn|ULBounds}
\end{eqnarray}
 \end{proposition}
\begin{proof}
If $m$ is odd, then the bound for $\mathcal U_m$ is attained when 
\eqref{eqn|dLanes} is applied to  $(m_1$, $m_2$; $\beta_1$, $\beta_2)$ $=$ $(\frac{m-1}2$, $\frac{m+1}2$; $1$, $0)$. The bound for $\mathcal L_m$ is attained for $(m_1$, $m_2$; $\beta_1$, $\beta_2)$ $=$ $(\frac{m+1}2$, $\frac{m-1}2$; $0$, $0)$.

If $m$ is even, then the bound for $\mathcal U_m$ is attained for $(m_1$, $m_2$; $\beta_1$, $\beta_2)$ $=$ $(\frac m2$, $\frac m2$, $0$, $0)$, while the bound for $\mathcal L_m$ is attained for $(m_1$, $m_2$; $\beta_1$, $\beta_2)$ $=$ $(\frac m2-1$, $\frac m2+1$; $1$, $0)$.
\end{proof}



\section{Almost sure bounds}
In this section we prove Theorem \ref{thm|generalUpperBounds}.
We will first prove the inequalities \eqref{eqn|gBoundFibonacci}. It suffices to prove the proposition below. 
\begin{proposition} Assume $\varphi$ is a random variable such that for every subset $T\subseteq W$ with $k$ elements we have $\partial_T\varphi\in [L,U]$. Then, the following inequality holds for every subset $S\subseteq W$ with $k+1$ elements. 
\begin{eqnarray}
\partial_S\varphi\in[L-U,U-L].
\label{eqn|gFibGen}
\end{eqnarray}
\end{proposition}
\begin{proof}
Let $s$ be an arbitrary element of $S$. Let $T=S\setminus \{s\}$. 
\begin{eqnarray*} \partial_S\varphi(\omega)&=&\partial_T\varphi(\sigma_s^b(\omega))-\partial_T\varphi(\sigma_s^a(\omega)).
\end{eqnarray*}
The result \eqref{eqn|gFibGen} immediately follows from the previous equality.
\end{proof}

\begin{theorem}\label{thm|AbsoluteValueBound}
Let $k\in W$ and let $S\subseteq W$ be a subset with at least two elements. The derivatives of the first-passage percolation time $f$ satisfy the following inequalities for every $\omega\in \Omega$.
\begin{eqnarray}
\partial_kf(\omega)&\in&[0,b-a]; \label{eqn|firstDerivativeBound}\\
\partial_{S}f(\omega)&\in&[-(b-a),b-a],\quad\mbox{if } |S|=2; \label{eqn|secondDerivativeBound}\\
\left|\partial_Sf(\omega)\right|&\leq& 2^{|S|-2}\cdot (b-a).\label{eqn|higherDerivativeBounds}
\end{eqnarray}
\end{theorem}
\begin{proof} The relation \eqref{eqn|firstDerivativeBound} is obvious because the function $f$ must increase, and it can increase by at most $b-a$ if one edge changes its passage time from $a$ to $b$. Observe that for sets $S$ with two elements, the relation \eqref{eqn|secondDerivativeBound} and the inequality \eqref{eqn|higherDerivativeBounds} follow directly from 
\eqref{eqn|firstDerivativeBound} and \eqref{eqn|gFibGen}.   
Observe that if $\varphi$ is any function, and not just first passage percolation time, then \eqref{eqn|LeibnitzRule} implies $|\partial_G\varphi(\omega)|\leq 2^{|G|} \|\varphi\|_{\infty}$ for every set $G$. Assume now that $S$ has at least two elements $k$ and $l$. Let $G=S\setminus\{k,l\}$. 
\begin{eqnarray*}
\left|\partial_Sf(\omega)\right|&=&\left|\partial_G\left(\partial_k\partial_lf(\omega)\right)\right|\leq 2^{|G|}\cdot \left\|\partial_k\partial_lf\right\|_{\infty}\\
&\leq& 2^{|G|}\cdot (b-a).
\end{eqnarray*} The proof is complete once we observe that $|G|=|S|-2$.
\end{proof}

\begin{theorem} \label{thm|Case2}
The values $\mathcal U_1$, $\mathcal L_1$, $\mathcal U_2$, and $\mathcal L_2$ are
given in the table \begin{eqnarray*}&&
\begin{array}{|c|r|r| }
\hline 
k&{\quad}1&{\quad}2\\
\hline
\mathcal U_k&1&1 \\
\hline 
\mathcal L_k&0&-1 \\
\hline
\end{array} 
\end{eqnarray*}
\end{theorem}
   \begin{proof} The inequality \eqref{eqn|firstDerivativeBound} implies that $\mathcal U_1\leq 1$ and $\mathcal L_1\geq 0$. Trivial examples establish the equalities 
   $\mathcal U_1=1$ and $\mathcal L_1=0$. 
For concreteness, take the source $(0,0)$ and the sink $(0,1)$ in $\mathbb Z^2$.
   If we assign $a$ to every edge and choose the edge $k$ to be anything other than the edge between the source and the sink, the derivative will be $0$. If we assign $b$ to every edge and choose $k$ to be the edge connecting the source and the sink, then the derivative will be $(b-a)$. 
   
 The case $k=2$ is the consequence of \eqref{eqn|secondDerivativeBound} and \eqref{eqn|ULBounds}. The inclusion \eqref{eqn|secondDerivativeBound} implies 
 $\mathcal U_2\leq 1$ and $\mathcal L_2\geq -1$. The inequalities \eqref{eqn|ULBounds} imply 
 $\mathcal U_2\geq 1$ and $\mathcal L_2\leq -1$.
   \end{proof}



\section{Evaluation of $\mathcal U_3$, $\mathcal L_3$, $\mathcal U_4$, and $\mathcal L_4$}

\subsection{Upper bounds}\label{subs|U3}
\begin{theorem}\label{thm|upperBoundS3}
Let $S\subseteq W$ be a subset with three elements. The first passage percolation time $f$ satisfies the following inequality for every $\omega\in\Omega$
\begin{eqnarray}
\partial_Sf(\omega)&\leq& (b-a). \label{eqn|upperBoundS3}
\end{eqnarray}
\end{theorem}
\begin{proof} Let $S=\{k,l,m\}$. We will make our notation shorter and write $\sigma^{(\theta_1,\theta_2,\theta_3)}(\omega)$ instead of $\sigma_k^{\theta_1}\circ \sigma_l^{\theta_2}\circ\sigma_m^{\theta_3}(\omega)$ for $(\theta_1,\theta_2,\theta_3)\in\{a,b\}^3$.  We will first prove that the inequality \eqref{eqn|upperBoundS3} is satisfied if 
\begin{eqnarray}\sigma^{(a,a,a)}(\omega) &\in& E_k^C\cup E_l^C\cup E_m^C.
\label{eqn|conditionU3AAA}
\end{eqnarray} 
If we assume $\sigma^{(a,a,a)}(\omega)\in E_k^C$, then Proposition \ref{thm|forComputer} (a) implies that 
$f(\sigma^{(b,a,a)}(\omega))$ and $f(\sigma^{(a,a,a)}(\omega))$ are equal, i.e. $\partial_kf(\sigma^{(a,a,a)}(\omega))=0$. Then,
\begin{eqnarray*}
\partial_Sf(\omega)&=& \partial_kf(\sigma^{(a,b,b)}(\omega)) 
-\partial_kf(\sigma^{(a,a,b)}(\omega))-\partial_kf(\sigma^{(a,b,a)}(\omega)) \\
&\leq&(b-a).
\end{eqnarray*} 
We have proved that $\sigma^{(a,a,a)}(\omega)\in E_k^C$ implies $\partial_Sf(\omega)\leq (b-a)$. In analogous ways we prove that 
the inequality \eqref{eqn|upperBoundS3} is implied if $\sigma^{(a,a,a)}(\omega)$ belongs to $E_l^C$ or $E_m^C$. 

We will now prove that \eqref{eqn|upperBoundS3} holds if \begin{eqnarray}\sigma^{(b,a,b)}(\omega)\not\in \hat E_k^C\cap E_l\cap \hat E_m^C.
\label{eqn|conditionViolatedU3BAB}
\end{eqnarray}

If $\sigma^{(b,a,b)}(\omega)\in E_l^C$, then 
$\partial_lf(\sigma^{(b,a,b)}(\omega))=0$,  
due to Proposition 
\ref{thm|forComputer} (a).  The derivative $\partial_Sf(\omega)$ becomes 
\begin{eqnarray*}
\partial_Sf(\omega)&=&-\partial_lf(\sigma^{(a,a,b)}(\omega))-
 \partial_lf(\sigma^{(b,a,a)}(\omega))+\partial_l f(\sigma^{(a,a,a)}(\omega)) \\
&\leq&-0-0+(b-a)=b-a.
\end{eqnarray*}
Assume now that $\sigma^{(b,a,b)}(\omega)\in \hat E_k$. Proposition \ref{thm|forComputer} (b) implies 
$\partial_kf(\sigma^{(a,a,b)}(\omega))=b-a$. Therefore,   
\begin{eqnarray*}\partial_Sf(\omega)&=&-(b-a) + \partial_kf(\sigma^{(a,b,b)}(\omega))-\partial_kf(\sigma^{(a,b,a)}(\omega)) +
\partial_kf(\sigma^{(a,a,a)}(\omega))\\
&\leq&-(b-a)+(b-a)-0+(b-a)=(b-a).
\end{eqnarray*}
Since the case $\sigma^{(b,a,b)}(\omega)\in \hat E_m$ is analogous, we have completed the proof of \eqref{eqn|conditionViolatedU3BAB}.
In analogous way we prove that $\partial_Sf(\omega)\leq (b-a)$ holds if either of the following two inclusions is satisfied:
\begin{eqnarray*} \sigma^{(a,b,b)}(\omega)\not\in E_k\cap \hat E_l^C\cap \hat E_m^C\quad\mbox{or}\quad
 \sigma^{(b,b,a)}(\omega)\not\in \hat E_k^C\cap \hat E_l^C\cap E_m.
\end{eqnarray*}
It remains to prove \eqref{eqn|upperBoundS3} if we assume that all of the following four conditions are satisfied 
\begin{eqnarray}
\sigma^{(a,b,b)}(\omega)&\in& E_k\cap\hat  E_l^C\cap \hat E_m^C,\label{eqn|noHelpfulABB}\\
\sigma^{(b,a,b)}(\omega)&\in& \hat E_k^C\cap E_l\cap \hat E_m^C,\label{eqn|noHelpfulBAB}\\
\sigma^{(b,b,a)}(\omega)&\in& \hat E_k^C\cap\hat  E_l^C\cap E_m,\label{eqn|noHelpfulBBA}\\
\sigma^{(a,a,a)}(\omega) &\in& E_k \cap E_l \cap E_m. 
\label{eqn|noHelpfulAAA}
\end{eqnarray} 

Let $\gamma$ be a geodesic on $\sigma^{(a,a,a)}(\omega)$. Since we assumed \eqref{eqn|noHelpfulAAA}, all of the edges $k$, $l$, and $m$ must belong to $\gamma$. Without loss of generality, assume that they appear in this order: $k$, $l$, $m$. Let $\gamma(b,a,b) $ be a geodesic on $\sigma^{(b,a,b)}(\omega)$. According to \eqref{eqn|noHelpfulBAB}, the geodesic $\gamma(b,a,b)$ must pass through $l$ and must not pass through either of $k$, $m$. Let $\gamma^-(b,a,b)$ and $\gamma^+(b,a,b)$ be the sections of $\gamma(b,a,b)$ before and after the edge $l$. The sections $\gamma^-(b,a,b)$ and $\gamma^+(b,a,b)$ are assumed not to contain the edge $l$. 
Let us denote by $T^-(b,a,b)$ and $T^+(b,a,b)$ the passage times over the sections $\gamma^-(b,a,b)$ and $\gamma^+(b,a,b)$.

When the geodesic  $\gamma$ passes through the edge  $l$, it must pass through both of its endpoints. One of these endpoints is encountered before the other. Let $L^-$ denote the initial endpoint and $L^+$ the terminal endpoint of the edge $l$.

There are two possible orders of $L^-$ and $L^+$ on the path $\gamma(b,a,b)$. The first possibility is that the endpoint $L^-$ appears before $L^+$ on $\gamma^{(b,a,b)}$. This possibility is shown with solid line in Figure \ref{fig|worstCaseUpperBound3}. The second possibility is that $L^+$ appears before $L^-$. The dashed line in Figure \ref{fig|worstCaseUpperBound3} represents the case in which this occurs.

 \begin{figure}[t]
 \centering
\begin{tikzpicture}
\draw[thick] (0.5,4) -- (11.5,4);
\node at (2.6,4) {$\bullet$};
\node at (3,4) {$\bullet$};
\node at (4.6,4) {$\bullet$};
\node at (5,4) {$\bullet$};
\node at (6.6,4) {$\bullet$};
\node at (7,4) {$\bullet$}; 
\draw[very thick] (2.6,4) -- (3,4);
\draw[very thick] (4.6,4) -- (5,4);
\draw[very thick] (6.6,4) -- (7,4); 

\node[above] at (2.8,4) {$k$};
\node[above] at (4.8,4) {$l$};
\node[above] at (6.8,4) {$m$}; 
\node[below] at (4.3,4) {$L^-$};
\node[below] at (5.3,4) {$L^+$};

\draw[thick] (0.5, 5.5) .. controls (2,5.5) and (4,5.5) .. (4.6,4);
\node[above] at (1.7,5.7) {$\gamma^-(b,a,b)$};

\draw[thick,dashed] (0.5, 6.5) .. controls (2,6.5) and (5,6.5) .. (5,4); 
 
\draw[thick] (5,4) .. controls (4.6,2) and (9,2.5) .. (11.5,2.5);
\node[below] at (10.5,2.3) {$\gamma^+(b,a,b)$};

\draw[thick,dashed] (4.6,4) .. controls (4.1,1.5) and (9,1.5) .. (11.5,1.5); 
 
\end{tikzpicture}
\caption{Configuration of paths described in the proof of Theorem \ref{thm|upperBoundS3}.}
\label{fig|worstCaseUpperBound3}
\end{figure}

Let us denote by $T^-$ the passage time on the environment $\sigma^{(a,a,a)}(\omega)$ over the segment of $\gamma$ between the source and just before reaching the edge $l$. Denote by $T^+$ the passage time on $\sigma^{(a,a,a)}(\omega)$ over the segment
after the edge $l$ until the sink. The following equations must hold
\begin{eqnarray}
f(\sigma^{(a,a,a)}(\omega))&=& T^-+T^++a,\label{eqn|U3TAAA}\\
f(\sigma^{(b,a,b)}(\omega))&=& T^-(b,a,b)+T^+(b,a,b)+a.\label{eqn|U3TBAB}.\end{eqnarray}
In the case when $L^-$ appears before $L^+$ on $\gamma(b,a,b)$, the following two inequalities are satisfied:
\begin{eqnarray}
f(\sigma^{(a,a,b)}(\omega))&\leq& T^-+T^+(b,a,b)+a,\label{eqn|U3TAAB}\\
f(\sigma^{(b,a,a)}(\omega))&\leq& T^-(b,a,b)+T^++a.\label{eqn|U3TBAA}
\end{eqnarray}
If $L^+$ appears before $L^-$, then even stronger relations hold:
\begin{eqnarray}
f(\sigma^{(a,a,b)}(\omega))&\leq& T^-+T^+(b,a,b),\label{eqn|U3TAABS}\\
f(\sigma^{(b,a,a)}(\omega))&\leq& T^-(b,a,b)+T^+.\label{eqn|U3TBAAS}
\end{eqnarray}
Clearly, \eqref{eqn|U3TAABS} and \eqref{eqn|U3TBAAS} imply \eqref{eqn|U3TAAB} and \eqref{eqn|U3TBAA} are satisfied. Hence,  \eqref{eqn|U3TAAB} and \eqref{eqn|U3TBAA} hold always.
The relations \eqref{eqn|U3TAAA}, \eqref{eqn|U3TBAB}, \eqref{eqn|U3TAAB}, and \eqref{eqn|U3TBAA} imply 
\begin{eqnarray} f(\sigma^{(a,a,b)}(\omega))+f(\sigma^{(b,a,a)}(\omega))-f(\sigma^{(b,a,b)}(\omega))-f(\sigma^{(a,a,a)}(\omega))
&\leq& 0.\label{eqn|U3HelpfulCancelation}
\end{eqnarray}
Therefore, the derivative $\partial_Sf(\omega)$ can be bounded in the following way
\begin{eqnarray}
\partial_Sf(\omega)&\leq& \partial_kf(\sigma^{(a,b,b)}(\omega))-\partial_kf(\sigma^{(a,b,a)}(\omega)).
\label{eqn|U3ReducedToTwo}
\end{eqnarray}
Since the derivatives of the first order belong to $[0,b-a]$, the first term on the right-hand side in
\eqref{eqn|U3ReducedToTwo} is smaller than or equal to $(b-a)$ while $\partial_kf(\sigma^{(a,b,a)}(\omega))$ is greater than or equal to $0$. Therefore, \eqref{eqn|U3ReducedToTwo}  implies \eqref{eqn|upperBoundS3}.
\end{proof}

\subsection{Direction switching}\label{subs|DSwitch}
We first prove that we cannot have two geodesics on the same environment that go in opposite directions over an edge.
\begin{proposition}\label{thm|DSwitchTrivial} Assume that $k$ is an edge and that $K_1$ and $K_2$ are the endpoints of $k$. Assume that $\omega$ is a fixed environment and that there are two geodesics $\mu_1$ and $\mu_2$ on $\omega$ that contain $k$. Then the order of $K_1$ and $K_2$ on $\mu_1$ must be the same as the order of $K_1$ and $K_2$ on $\mu_2$. \end{proposition}
\begin{proof}

 \begin{figure}[t]
 \centering
\begin{tikzpicture}
\draw[thick,dashed] (0,2) -- (5,2);
\draw[thick,dotted] (5,2)--(10,2);
\draw[thick,dotted] (0,0)--(5,0);
\draw[thick,dashed] (5,0)--(10,0); 
\draw[very thick] (5,0)--(5,2);
\node[above] at (5,2) {$K_1$};
\node[below] at (5,0) {$K_2$};
\node[right] at (5,1) {$k$};
\node[above] at (1,2) {$\mu_1^-$};
\node[above] at (9,2) {$\mu_2^+$};
\node[below] at (1,0) {$\mu_2^-$};
\node[below] at (9,0) {$\mu_1^+$}; 
\node at (5,2) {$\bullet$};
\node at (5,0) {$\bullet$}; 
\end{tikzpicture}
\caption{Configuration described in proof of Proposition \ref{thm|DSwitchTrivial}.}
\label{fig|DSwitchingTrivial}
\end{figure}

Assume the contrary, that the order of $K_1$ and $K_2$ is different on $\mu_1$ and $\mu_2$. Assume that $K_1$ appears before $K_2$ on $\mu_1$ and that $K_2$ appears before $K_1$ on $\mu_2$. Let $\mu_1^-$ be the section of $\mu_1$ from the source to the vertex $K_1$ and $\mu_1^+$ the section of $\mu_1$ from the vertex $K_2$ to the sink. Let $\mu_2^-$ be the section of $\mu_2$ from the source to $K_2$ and $\mu_2^+$ the section of $\mu_2$ from $K_1$ to the sink. 
The paths $\mu_1^-\cup \mu_2^+$ and $\mu_2^-\cup \mu_1^+$ are two paths between the source and the sink. Their passage times are at least as big as the passage times over the geodesics $\mu_1$ and $\mu_2$. Therefore, \begin{eqnarray} T(\mu_1) \leq T(\mu_1^-\cup \mu_2^+)\quad\mbox{and}\quad T(\mu_2)\leq T(\mu_2^-\cup \mu_1^+).
\label{eqn|DSwitchTrivialmuIneqs}\end{eqnarray}
 Moreover,
\begin{eqnarray}
T(\mu_1)&=& T(\mu_1^-)+T(\mu_1^+)+\omega_k, \label{eqn|DSwitchTrivialmu1}\\ 
T(\mu_2)&=& T(\mu_2^-)+T(\mu_2^+)+\omega_k, \label{eqn|DSwitchTrivialmu2}\\
T(\mu_1^-\cup \mu_2^+)&=& T(\mu_1^-)+T(\mu_2^+), \label{eqn|DSwitchTrivialmu1mmu2p}\\
T(\mu_2^-\cup \mu_1^+)&=& T(\mu_2^-)+T(\mu_1^+). \label{eqn|DSwitchTrivialmu2mmu1p}
\end{eqnarray}
If we add the two inequalities from \eqref{eqn|DSwitchTrivialmuIneqs} and then use \eqref{eqn|DSwitchTrivialmu1}, \eqref{eqn|DSwitchTrivialmu2}, \eqref{eqn|DSwitchTrivialmu1mmu2p}, and \eqref{eqn|DSwitchTrivialmu2mmu1p}, we obtain that $2\omega_k\leq 0$. This is a contradiction because $\omega_k\in\{a,b\}$ and must be strictly positive.
\end{proof}

In the next proposition we will prove that the direction of the flow cannot switch if only one edge is flipped from $a$ to $b$. 

\begin{proposition}\label{thm|DSwitch2} Assume that $k$ and $x$ are two edges and 
that $K_1$ and $K_2$ are the endpoints of $k$.
Assume that $\omega$ is a fixed environment such that 
on $\sigma_x^a(\omega)$ there is a geodesic 
$\lambda_a$ that contains $k$ and on $\sigma_x^b(\omega)$ there is a geodesic $\lambda_b$ that contains $k$.
Then, the order of $K_1$ and $K_2$ on $\lambda_a$ is the same as their order on $\lambda_b$. 
\end{proposition} 
\begin{proof}  We will only consider the case $k\neq x$. The case $k=x$ is easier, and it will be discussed in the remark below the proof.
Assume the contrary, that $K_1$ appears before $K_2$ on $\lambda_a$, but $K_2$ appears before $K_1$ on $\lambda_b$. 
Let $\lambda_a^-$  be the section of $\lambda_a$ before $K_1$. Let $\lambda_a^+$ be the section of $\lambda_a$ after $K_2$. Define $\lambda_b^-$ and $\lambda_b^+$ as the sections of $\lambda_b$ that appear before $K_2$ and after $K_1$, respectively. 

We first show that \(x \notin \lambda_b\). Suppose, for the sake of contradiction, that \(x \in \lambda_b\). Even though the value \(b\) is assigned to the edge \(x\), the geodesic \(\lambda_b\) passes through \(x\). On the environment $\sigma_x^a(\omega)$, when the value \(a\) is assigned to \(x\), the path \(\lambda_b\) must therefore still be a geodesic. Hence, \(\lambda_a\) and \(\lambda_b\) are both geodesics in \(\sigma_x^a(\omega)\), contradicting Proposition~\ref{thm|DSwitchTrivial}.

Therefore, $x\not\in\lambda_b$. We will now prove that we must have $x\in\lambda_a$. Assume that $x\not \in \lambda_a$. This means that even though $a$ is assigned to $x$, there is a geodesic $\lambda_a$ that omits $x$. Hence, if $b$ is assigned to $x$, then $\lambda_a$ will be a geodesic as well. We now have that $\lambda_a$ and $\lambda_b$ are geodesics on $\sigma_x^b(\omega)$, which is impossible because of Proposition \ref{thm|DSwitchTrivial}.

Hence, $x\in\lambda_a$. Without loss of generality, assume that $x\in\lambda_a^-$, as shown in Figure \ref{fig|DSwitching2}.
 \begin{figure}[t]
 \centering
\begin{tikzpicture}
\draw[thick,dashed] (0,2) -- (5,2);
\draw[thick,dotted] (5,2)--(10,2);
\draw[thick,dotted] (0,0)--(5,0);
\draw[thick,dashed] (5,0)--(10,0);
\draw[very thick] (2,2)--(4,2); 
\draw[very thick] (5,0)--(5,2);
\node[above] at (5,2) {$K_1$};
\node[below] at (5,0) {$K_2$};
\node[right] at (5,1) {$k$};
\node[above] at (1,2) {$\lambda_a^-$};
\node[above] at (9,2) {$\lambda_b^+$};
\node[below] at (1,0) {$\lambda_b^-$};
\node[below] at (9,0) {$\lambda_a^+$};
\node at (2,2) {$\bullet$};
\node at (4,2) {$\bullet$}; 
\node at (5,2) {$\bullet$};
\node at (5,0) {$\bullet$};
\node[above] at (3,2) {$x$}; 
\end{tikzpicture}
\caption{Configuration described in proof of Proposition \ref{thm|DSwitch2}.}
\label{fig|DSwitching2}
\end{figure}

The value $\omega_k$ from $\{a,b\}$ that is assigned to the edge $k$ is the same on $\sigma_x^a(\omega)$ and $\sigma_x^b(\omega)$.
The path $\lambda_a=\lambda_a^-\cup\{k\}\cup \lambda_a^+$ is a geodesic on $\sigma_x^a(\omega)$, while $\lambda_a^-\cup\lambda_b^+$ is just a path from the source to the sink. Therefore, we must have 
\begin{eqnarray}
\omega_k+T(\lambda_a^+)&\leq& T(\lambda_b^+). \label{eqn|Switch20a}
\end{eqnarray}
The path $\lambda_b$ is a geodesic on $\sigma_x^b(\omega)$. The path $\lambda_b^-\cup \lambda_a^+$ does not have to be a geodesic, hence 
\begin{eqnarray}
\omega_k+T(\lambda_b^+)&\leq& T(\lambda_a^+). \label{eqn|Switch20b}
\end{eqnarray}
If we add \eqref{eqn|Switch20a} and \eqref{eqn|Switch20b}, we obtain $2\omega_k\leq 0$. This is a contradiction. 
\end{proof}
\noindent{\em Remark.} In the case $k=x$, we can skip all the case-work needed to establish that $x\in\lambda_a$. We just make a slight modification to the proof by replacing $\omega_k$ by $a$ in \eqref{eqn|Switch20a} and by $b$ in \eqref{eqn|Switch20b}. The conclusion $a+b\leq 0$ leads to contradiction in this case.

\begin{definition}\label{def|DSwitch3} Assume that $k$, $l$, and $m$ are three different fixed edges. 
The edge $k$ is called a direction switching edge with respect to the pair of edges $(l,m)$ on the environment $\omega$ if the two endpoints $K_1$ and $K_2$ of $k$ satisfy:
On the environment $ \sigma_l^a\sigma_m^b(\omega)$ there is a geodesic such that $K_1$ is before $K_2$; while on the environment 
$ \sigma_l^b\sigma_m^a(\omega)$ there is a geodesic such that $K_2$ is before $K_1$. 
\end{definition}

\begin{proposition}\label{thm|S3DSwitchConfiguration} Assume that $S=\{k,l,m\}$ and that $k$ is a direction switching edge with respect to $(l,m)$. 
Let $K_1$ and $K_2$ be the endpoints of $k$. Let $\lambda$ and $\mu$ be two fixed paths. Assume that $\lambda$ is a geodesic on $\sigma_l^a\sigma_m^b(\omega)$ and that $K_1$ appears before $K_2$ on $\lambda$. Assume that $\mu$ is a geodesic on $\sigma_l^b\sigma_m^a(\omega)$ and that $K_2$ appears before $K_1$ on $\mu$. Denote by $\lambda^-$ the section of $\lambda$ between the source and $K_1$ and by $\lambda^+$ the section of $\lambda$ between $K_2$ and the sink. Similarly, let $\mu^-$ and $\mu^+$ be the sections of $\mu$ before and after the edge $k$.  
Then, one of the following two mathematical propositions $P(\lambda^-,\mu^+)$ and $P(\lambda^+,\mu^-)$ is satisfied
\begin{eqnarray}
P(\lambda^-,\mu^+)&\equiv&\left\{ (l\in \lambda^-) \quad\mbox{and}\quad (m\in\mu^+)\right\},\label{eqn|defLmMp}\\
P(\lambda^+,\mu^-)&\equiv&\left\{ (l\in \lambda^+) \quad\mbox{and}\quad (m\in\mu^-)\right\}.\label{eqn|defLpMm}
\end{eqnarray}
\end{proposition}
\begin{proof} Due to Propositions \ref{thm|DSwitchTrivial} and \ref{thm|DSwitch2}, we must have $\{l,m\}\subseteq\lambda\cup \mu$.

Let us prove that we can't have $\{l,m\}\subseteq\lambda$ or $\{l,m\}\subseteq \mu$. 
Assume the contrary, that both of $l$ and $m$ belong to $\lambda$. 
We have four cases: when $\{l,m\}\subseteq \lambda^-$; $\{l,m\}\subseteq \lambda^+$; $l\subseteq \lambda^-$ and $m\in\lambda^+$; and  $l\subseteq \lambda^+$ and $m\in\lambda^-$. 

The first two cases are simpler and analogous. The passage times over $\lambda^-$, $\lambda^+$, $\mu^-$, and $\mu^+$ remain the same if the environment changes from $\sigma_l^a\sigma_m^b(\omega)$ to
$\sigma_l^b\sigma_m^a(\omega)$. Therefore, both $\mu$ and $\lambda$ are geodesics on each of the two environments. 
The path $\lambda^-\cup \mu^+$ has passage time at least as big as the geodesic $\lambda$, hence
\[T(\lambda^-)+T(\mu^+)
\geq T(\lambda^-)+\omega_k+T(\lambda^+).\]
In an analogous way we obtain 
\[T(\mu^-)+T(\lambda^+)\geq T(\mu^-)+\omega_k+T(\mu^+).\]
If we add the last two inequalities, we obtain $0\geq 2\omega_k$, which is impossible because $\omega_k\in\{a,b\}$ and $b>a>0$.

 \begin{figure}[t]
 \centering
\begin{tikzpicture}
\draw[thick,dashed] (0,2) -- (5,2);
\draw[thick,dotted] (5,2)--(10,2);
\draw[thick,dotted] (0,0)--(5,0);
\draw[thick,dashed] (5,0)--(10,0);
\draw[very thick] (2,2)--(4,2);
\draw[very thick] (6,0)--(8,0);
\draw[very thick] (5,0)--(5,2);
\node[above] at (5,2) {$K_1$};
\node[below] at (5,0) {$K_2$};
\node[right] at (5,1) {$k$};
\node[above] at (1,2) {$\lambda^-$};
\node[above] at (9,2) {$\mu^+$};
\node[below] at (1,0) {$\mu^-$};
\node[below] at (9,0) {$\lambda^+$};
\node at (2,2) {$\bullet$};
\node at (4,2) {$\bullet$};
\node at (6,0) {$\bullet$};
\node at (8,0) {$\bullet$};
\node at (5,2) {$\bullet$};
\node at (5,0) {$\bullet$};
\node[above] at (3,2) {$l$};
\node[below] at (7,0) {$m$};
\end{tikzpicture}
\caption{Case $l\in\lambda^-$ and $m\in \lambda^+$.}
\label{fig|LLmMLP}
\end{figure}
If we assume that $l\in \lambda^-$ and $m\in \lambda^+$, then we have a situation as shown in Figure \ref{fig|LLmMLP}. 
We start by analyzing the environment $\sigma_l^a\sigma_m^b(\omega)$. The passage over $
\lambda$ must be smaller than or equal to the passage over $\lambda^-\cup \mu^+$, i.e.
\begin{eqnarray} \omega_k+b+T(\lambda^+\setminus m)\leq T(\mu^+). \label{eqn|DSwitch2BothOnLambda01}\end{eqnarray} 
On the environment $\sigma_l^b\sigma_m^a(\omega)$, the passage over 
$\mu$ is smaller than or equal to the passage time over $\mu^-\cup \lambda^+$, which give us 
\begin{eqnarray}\omega_k+T(\mu^+)&\leq &a+T(\lambda^+\setminus\{m\}). \label{eqn|DSwitch2BothOnLambda02}\end{eqnarray}
If we add \eqref{eqn|DSwitch2BothOnLambda01} and
\eqref{eqn|DSwitch2BothOnLambda02} and then subtract $T(\lambda^+\setminus\{m\})+T(\mu^+)$ from both sides, we obtain 
$2\omega_k+b\leq a$. This is a contradiction
 because $\omega_k\in\{a,b\}$ and $b>a>0$.

 \begin{figure}[t]
 \centering
\begin{tikzpicture}
\draw[thick,dashed] (0,2) -- (5,2);
\draw[thick,dotted] (5,2)--(10,2);
\draw[thick,dotted] (0,0)--(5,0);
\draw[thick,dashed] (5,0)--(10,0);
\draw[very thick] (2,2)--(4,2);
\draw[very thick] (6,0)--(8,0);
\draw[very thick] (5,0)--(5,2);
\node[above] at (5,2) {$K_1$};
\node[below] at (5,0) {$K_2$};
\node[right] at (5,1) {$k$};
\node[above] at (1,2) {$\lambda^-$};
\node[above] at (9,2) {$\mu^+$};
\node[below] at (1,0) {$\mu^-$};
\node[below] at (9,0) {$\lambda^+$};
\node at (2,2) {$\bullet$};
\node at (4,2) {$\bullet$};
\node at (6,0) {$\bullet$};
\node at (8,0) {$\bullet$};
\node at (5,2) {$\bullet$};
\node at (5,0) {$\bullet$};
\node[above] at (3,2) {$m$};
\node[below] at (7,0) {$l$};
\end{tikzpicture}
\caption{Case $m\in\lambda^-$ and $l\in \lambda^+$.}
\label{fig|LLPMLM}
\end{figure}

The case $m\in \lambda^-$ and $l\in\lambda^+$ is very similar. It corresponds to Figure \ref{fig|LLPMLM}.
On $\sigma_l^a\sigma_m^b(\omega)$, we use that $T(\lambda)\leq T(\mu^-\cup \lambda^+)$ to obtain 
\begin{eqnarray}
T(\lambda^-\setminus\{m\})+b+\omega_k&\leq& T(\mu^-). \label{eqn|DSwitch2BothOnLambda03}
\end{eqnarray}
On $\sigma_l^b\sigma_m^a(\omega)$, the passage time over $\mu$ must be smaller than or equal to the passage time over $\lambda^-\cup \mu^+$. The inequality 
$T(\mu)\leq T(\lambda^-\cup \mu^+)$ is equivalent to 
\begin{eqnarray}
T(\mu^-)+\omega_k&\leq& T(\lambda^-\setminus\{m\})+a. \label{eqn|DSwitch2BothOnLambda04}
\end{eqnarray}
Adding \eqref{eqn|DSwitch2BothOnLambda03} and \eqref{eqn|DSwitch2BothOnLambda04} and subtracting $ T(\lambda^-\setminus\{m\})+T(\mu^-)$ implies
$2\omega_k+b\leq a$, which is not possible.

We proved that $\{l,m\}\subseteq \lambda\cup \mu$, $\{l,m\}\not\subseteq\lambda$, and $\{l,m\}\not\subseteq \mu$.

Therefore, $\lambda$ contains exactly one of the edges $\{l,m\}$ and $\mu$ contains the other one.  Let us prove that we can't have $m\in\lambda$ and $l\in\mu$. Assume the contrary. The path $\lambda$ is a geodesic when the value $a$ is assigned to $l$ and the value $b$ is assigned to $m$; while $\mu$ is a geodesic when $a$ is assigned to $m$ and $b$ is assigned to $l$. Since we assumed $m\in\lambda\setminus \mu$ and $l\in\mu\setminus \lambda$, we have 
\begin{eqnarray*} f(\sigma_l^a\sigma_m^b(\omega))&=& T(\lambda, \sigma_l^a\sigma_m^b(\omega))=T(\lambda,  \sigma_m^b(\omega)).\end{eqnarray*}
In the last equality we used that $l\not\in \lambda$. We now use that $m\in \lambda$ to further conclude \begin{eqnarray*}
f(\sigma_l^a\sigma_m^b(\omega))&=& T(\lambda,  \sigma_m^a(\omega))+(b-a)=T(\lambda, \sigma_l^b \sigma_m^a(\omega))+(b-a).
\end{eqnarray*}
Again, in the last equality we used that $l\not\in \lambda$. Now we use the last equation and the fact that $\mu$ is a geodesic on the environment 
$\sigma_l^b\sigma_m^a(\omega)$ to obtain
\begin{eqnarray*}
f(\sigma_l^a\sigma_m^b(\omega))
&\geq& f( \sigma_l^b \sigma_m^a(\omega))+(b-a)=T(\mu,\sigma_l^b \sigma_m^a(\omega))+(b-a).
\end{eqnarray*}
Since $m\not\in \mu$ and $l\in \mu$, the last inequality implies 
\begin{eqnarray*}
f(\sigma_l^a\sigma_m^b(\omega))
&\geq&T(\mu,\sigma_l^b (\omega))+(b-a) =T(\mu,\sigma_l^a (\omega))+2(b-a).\end{eqnarray*}
We use that $m\not \in \mu$ to further transform the last inequality into 
\begin{eqnarray*}
f(\sigma_l^a\sigma_m^b(\omega))
&\geq&T(\mu,\sigma_l^a\sigma_m^b (\omega))+2(b-a)\\
&\geq&f(\sigma_l^a\sigma_m^b(\omega))+2(b-a).
\end{eqnarray*}
This is a contradiction because $b>a$. 
Therefore, $l\in\lambda$ and $m\in\mu$.

In order to prove that $P(\lambda^-,\mu^+)$ or $P(\lambda^+,\mu^-)$ is satisfied, we must prove that neither of the following two propositions $P(\lambda^-,\mu^-)$, $P(\lambda^+,\mu^+)$ can hold.
\begin{eqnarray}
P(\lambda^-,\mu^-)&\equiv& \left\{(l\in \lambda^-) \quad\mbox{and}\quad (m\in\mu^-)\right\},\label{eqn|defLmMm}\\
P(\lambda^+,\mu^+)&\equiv&\left\{ (l\in \lambda^+) \quad\mbox{and}\quad (m\in\mu^+) \right\}.\label{eqn|defLpMp}
\end{eqnarray}

 \begin{figure}[t]
 \centering
\begin{tikzpicture}
\draw[thick,dashed] (0,2) -- (5,2);
\draw[thick,dotted] (5,2)--(10,2);
\draw[thick,dotted] (0,0)--(5,0);
\draw[thick,dashed] (5,0)--(10,0);
\draw[very thick] (2,2)--(4,2);
\draw[very thick] (2,0)--(4,0);
\draw[very thick] (5,0)--(5,2);
\node[above] at (5,2) {$K_1$};
\node[below] at (5,0) {$K_2$};
\node[right] at (5,1) {$k$};
\node[above] at (1,2) {$\lambda^-$};
\node[above] at (9,2) {$\mu^+$};
\node[below] at (1,0) {$\mu^-$};
\node[below] at (9,0) {$\lambda^+$};
\node at (2,2) {$\bullet$};
\node at (4,2) {$\bullet$};
\node at (2,0) {$\bullet$};
\node at (4,0) {$\bullet$};
\node at (5,2) {$\bullet$};
\node at (5,0) {$\bullet$};
\node[above] at (3,2) {$l$};
\node[below] at (3,0) {$m$};
\end{tikzpicture}
\caption{Geodesics $\lambda$ and $\mu$ if the proposition $P(\lambda^-,\mu^-)$ is satisfied.}
\label{fig|PLmMm}
\end{figure}

Assume that $P(\lambda^-,\mu^-)$ is satisfied. This configuration is shown in Figure \ref{fig|PLmMm}.
Then, since $\lambda^-\cup \mu^+$ is not a geodesic on the environment $\sigma_l^a\sigma_m^b(\omega)$, while $\lambda$ is, we obtain 
\begin{eqnarray}
T(\mu^+)&\geq &
\omega_k+T(\lambda^+). \label{eqn|PLmMmE1}
\end{eqnarray}
Since $\mu^-\cup \lambda^+$ is not a geodesic on  $\sigma_l^b\sigma_m^a(\omega)$, while $\mu$ is, we derive
\begin{eqnarray}
T(\lambda^+)&\geq&
\omega_k+T(\mu^+). \label{eqn|PLmMmE2}
\end{eqnarray}
Adding the inequalities \eqref{eqn|PLmMmE1} and  \eqref{eqn|PLmMmE2} implies 
$0>2\omega_k$, which is a contradiction.
In an analogous way we treat the case in which $P(\lambda^+,\mu^+)$ holds. 

Thus, we have proved that one of the mathematical propositions 
$P(\lambda^-,\mu^-)$, 
$P(\lambda^+,\mu^+)$ must be satisfied.
\end{proof}

\begin{proposition}\label{thm|dSwitchImpliesOmegaKEqA}
Assume that $S=\{k,l,m\}$ and that for fixed environment $\omega$, $k$ is a direction switching edge with respect to 
$(l,m)$ on $\omega$. Then, $\omega_k=a$ and the numbers $a$ and $b$ must satisfy 
\begin{eqnarray} b&\geq&3a.\label{eqn|BBigger3a}
\end{eqnarray}
\end{proposition}
\begin{proof} 
Let us define $\lambda$, $\mu$, $\lambda^{\pm}$, and $\mu^{\pm}$ as in Proposition \ref{thm|S3DSwitchConfiguration}. The same proposition allows us to assume that $l\in \lambda^-$ and $m\in \mu^+$ as in Figure \ref{fig|DSwitchNonTrivialCase}.

 \begin{figure}[t]
 \centering
\begin{tikzpicture}
\draw[thick,dashed] (0,2) -- (5,2);
\draw[thick,dotted] (5,2)--(10,2);
\draw[thick,dotted] (0,0)--(5,0);
\draw[thick,dashed] (5,0)--(10,0);
\draw[very thick] (2,2)--(4,2);
\draw[very thick] (6,2)--(8,2);
\draw[very thick] (5,0)--(5,2);
\node[above] at (5,2) {$K_1$};
\node[below] at (5,0) {$K_2$};
\node[right] at (5,1) {$k$};
\node[above] at (1,2) {$\lambda^-$};
\node[above] at (9,2) {$\mu^+$};
\node[below] at (1,0) {$\mu^-$};
\node[below] at (9,0) {$\lambda^+$};
\node at (2,2) {$\bullet$};
\node at (4,2) {$\bullet$};
\node at (6,2) {$\bullet$};
\node at (8,2) {$\bullet$};
\node at (5,2) {$\bullet$};
\node at (5,0) {$\bullet$};
\node[above] at (3,2) {$l$};
\node[above] at (7,2) {$m$};
\end{tikzpicture}
\caption{Geodesics $\lambda$ and $\mu$ under the assumptions $l\in\lambda^-$ and $m\in \mu^+$.}
\label{fig|DSwitchNonTrivialCase}
\end{figure}
Let us get some consequences of the fact that $\lambda$ is a geodesic on 
$\sigma_l^a\sigma_m^b(\omega)$. The path $\mu^-\cup \lambda^+$ cannot have a shorter passage time than $\lambda$. Hence, 
the passage time over $\lambda^-\cup \{k\}$ must be smaller than or equal to the passage time over $\mu^-$, i.e. 
\begin{eqnarray}T(\mu^-)&\geq& T(\lambda^-\setminus\{l\})+a+\omega_k.
\label{eqn|bgeq2ai01}
\end{eqnarray} 
On the environment
$\sigma_l^b\sigma_m^a(\omega)$, the path $\mu$ is a geodesic, while $\lambda^-\cup\mu^+$ does not have to be. Hence, the section $\mu^-\cup\{k\}$ has shorter passage time than $\lambda^-$, when the value $b$ is assigned to $l$. Therefore, 
\begin{eqnarray}
T(\mu^-)+\omega_k&\leq& T(\lambda^-\setminus\{l\})+b. 
\label{eqn|bgeq2ai02}
\end{eqnarray}
The inequalities \eqref{eqn|bgeq2ai01} and \eqref{eqn|bgeq2ai02} imply $b\geq a+2\omega_k$. This inequality is impossible if $\omega_k=b$. Thus, $\omega_k=a$ and 
\eqref{eqn|BBigger3a} must hold.
\end{proof}

\begin{proposition}\label{thm|lowerBoundS3DSwitch} 
Assume that $S=\{k,l,m\}$ and that for fixed environment $\omega$, $k$ is a direction switching edge with respect to 
$(l,m)$.
The following inequality must hold
\begin{eqnarray} 
\partial_Sf(\omega)&\geq&3a-b.\label{eqn|BoundS3DirectionSwitch}
\end{eqnarray}
\end{proposition}
\begin{proof} According to Proposition \ref{thm|dSwitchImpliesOmegaKEqA}, we know that $\omega_k=a$. 
We will write $\sigma^{\overrightarrow{\xi}}(\omega)$ instead of 
$\sigma^{\overrightarrow{\xi}}_{(k,l,m)}(\omega)$ for $\overrightarrow{\xi}\in\{a,b\}^3$. 
The following two inequalities are obvious because they follow from $\partial_zf(\omega)\geq 0$ for every edge $z$.
\begin{eqnarray}
f(\sigma^{(b,b,b)} (\omega))\geq f(\sigma^{(a,b,b)} (\omega)) &\mbox{and}& 
f(\sigma^{(b,a,a)} (\omega))\geq f(\sigma^{(a,a,a)} (\omega)). 
\label{eqn|LBS3DSwitchObv0102}
\end{eqnarray}
Let us define $\lambda$, $\mu$, $\lambda^{\pm}$, and $\mu^{\pm}$ as in Proposition \ref{thm|S3DSwitchConfiguration}. Without loss of generality, due to Proposition \ref{thm|S3DSwitchConfiguration}, we may assume that $l\in \lambda^-$ and $m\in \mu^+$ as in Figure \ref{fig|DSwitchNonTrivialCase}. Assume that $\omega_k=b$.

Since $\lambda$ is a geodesic on $\sigma^{(a,a,b)}(\omega)$, we have 
\begin{eqnarray}
f(\sigma^{(a,a,b)}(\omega))&=& T(\lambda^-\setminus\{l\})+T(\lambda^+)+2a.
\label{eqn|LBS3DSwitchL}
\end{eqnarray}
In a similar way we obtain 
\begin{eqnarray}
f(\sigma^{(a,b,a)}(\omega))&=&T(\mu^-)+ T(\mu^+\setminus\{m\})+2a.
\label{eqn|LBS3DSwitchM}
\end{eqnarray}
The passage times over $\lambda^-\cup \mu^+$ on the environments $\sigma^{(b,a,b)}(\omega)$ and $\sigma^{(b,b,a)}(\omega)$ are the same because the path $\lambda^-\cup \mu^+$ contains both $l$ and $m$. We will use $T_1$ to denote these two passage times, i.e. $T_1=T(\lambda^-\cup\mu^+,\sigma^{(b,a,b)}(\omega))$. 

The passage times over $\mu^-\cup \lambda^+$ are also the same on $\sigma^{(b,a,b)}(\omega)$ and $\sigma^{(b,b,a)}(\omega)$  because $\mu^-\cup \lambda^+$ contains neither $l$ nor $m$. We will use $T_2$ to denote these passage times.

The shortest passage times on $\sigma^{(b,a,b)}(\omega)$ and $\sigma^{(b,b,a)}(\omega)$ are bounded above by the number $\min\{T_1,T_2\}$. The numbers $T_1$ and $T_2$ satisfy 
\begin{eqnarray}T_1&=&T(\lambda^-\setminus\{l\})+T(\mu^+\setminus\{m\})+a+b,\label{eqn|LBS3DSwitchT1}\\
T_2&=&T(\mu^-)+T(\lambda^+). 
\label{eqn|LBS3DSwitchT2}
\end{eqnarray}

We now use $f(\sigma^{(b,a,b)}(\omega))\leq \min\{T_1,T_2\}$, $f(\sigma^{(b,b,a)}(\omega))\leq \min\{T_1,T_2\}$, \eqref{eqn|LBS3DSwitchObv0102}, \eqref{eqn|LBS3DSwitchL},  \eqref{eqn|LBS3DSwitchM}, \eqref{eqn|LBS3DSwitchT1}, and \eqref{eqn|LBS3DSwitchT2} to find a lower bound for $\partial_Sf(\omega)$. First, we use \eqref{eqn|LBS3DSwitchObv0102} to obtain 
\begin{eqnarray*}\partial_Sf(\omega)&\geq& -f(\sigma^{(b,a,b)}(\omega))-f(\sigma^{(b,b,a)}(\omega))+
f(\sigma^{(a,a,b)}(\omega))+f(\sigma^{(a,b,a)}(\omega)).\end{eqnarray*}
Next, we use $f(\sigma^{(b,a,b)}(\omega))\leq \min\{T_1,T_2\}$ and $f(\sigma^{(b,b,a)}(\omega))\leq \min\{T_1,T_2\}$ to further bound $\partial_Sf(\omega)$ as 
\begin{eqnarray*}\partial_Sf(\omega)&\geq& 
f(\sigma^{(a,a,b)}(\omega))+f(\sigma^{(a,b,a)}(\omega))-2\min\{T_1,T_2\}.\end{eqnarray*}
The equalities  \eqref{eqn|LBS3DSwitchL}, and \eqref{eqn|LBS3DSwitchM} further imply
\begin{eqnarray*}\partial_Sf(\omega)&\geq&
4a+T(\lambda^-\setminus\{l\})+ T(\lambda^+)+T(\mu^-)+T(\mu^+\setminus\{m\})-2 \min\{T_1,T_2\}.
\end{eqnarray*}
Finally, the last inequality, \eqref{eqn|LBS3DSwitchT1}, and \eqref{eqn|LBS3DSwitchT2} imply
\begin{eqnarray}\partial_Sf(\omega)&\geq & 3a-b +T_1+T_2-2\min\{T_1,T_2\}.\nonumber
\end{eqnarray}
It remains to observe that $T_1+T_2\geq 2\min\{T_1,T_2\}$, hence \eqref{eqn|BoundS3DirectionSwitch} is established.
\end{proof}

\subsection{Lower bounds}\label{subs|L3}
\begin{theorem}\label{thm|lowerBoundS3}
Let $S\subseteq W$ be a subset with three elements. The first passage percolation time $f$ satisfies the following inequality for every $\omega\in\Omega$
\begin{eqnarray}
\partial_Sf(\omega)&\geq& -(b-a). \label{eqn|lowerBoundS3}
\end{eqnarray}
\end{theorem}
\begin{proof} Let $S=\{k,l,m\}$. Due to Proposition \ref{thm|lowerBoundS3DSwitch} and $3a-b> -(b-a)$, we may assume that none of the edges is direction-switching with respect to the other two. 

We will first prove the following implication
\begin{eqnarray}
\sigma^{(a,a,b)}(\omega)\not\in E_k\cap E_l\cap \hat E_m^C&\Longrightarrow& \partial_Sf(\omega)\geq -(b-a)
\label{eqn|lowerBoundS3ImplicationMain}
\end{eqnarray}
The result \eqref{eqn|lowerBoundS3ImplicationMain} will follow from the following two
\begin{eqnarray}
\sigma^{(a,a,b)}(\omega)\in E_k^C &\Longrightarrow& \partial_Sf(\omega)\geq -(b-a),\label{eqn|lowerBoundS3ImplicationA}\\
\sigma^{(a,a,b)}(\omega)\in \hat E_m &\Longrightarrow& \partial_Sf(\omega)\geq -(b-a).\label{eqn|lowerBoundS3ImplicationB}
\end{eqnarray}
The first step in proving \eqref{eqn|lowerBoundS3ImplicationA} is to express the derivative $\partial_Sf(\omega)$ 
as 
\begin{eqnarray}
\partial_Sf(\omega)&=& \left(f(\sigma^{(b,b,b)}(\omega))-f(\sigma^{(b,b,a)}(\omega))\right) \label{eqn|lbS3ImpADer01}\\
&& +\left(f(\sigma^{(b,a,a)}(\omega))-
f(\sigma^{(a,a,a)}(\omega))\right) \label{eqn|lbS3ImpADer02}\\
&&-\left(f(\sigma^{(b,a,b)}(\omega))-f(\sigma^{(a,a,b)}(\omega))\right) \label{eqn|lbS3ImpADer03}\\
&&-\left(f(\sigma^{(a,b,b)}(\omega))-f(\sigma^{(a,b,a)}(\omega))
\right). \label{eqn|lbS3ImpADer04}
\end{eqnarray}
Assume that $\sigma^{(a,a,b)}(\omega)\in E_k^C$. Proposition \ref{thm|forComputer} (a) implies that the term \eqref{eqn|lbS3ImpADer03} 
is equal to $0$. The terms \eqref{eqn|lbS3ImpADer01} and \eqref{eqn|lbS3ImpADer02} are non-negative, and the negative term \eqref{eqn|lbS3ImpADer04} is bounded below
by $-(b-a)$, which proves \eqref{eqn|lowerBoundS3ImplicationA}.
In order to prove \eqref{eqn|lowerBoundS3ImplicationB}, we start by expressing $\partial_Sf(\omega)$ as 
\begin{eqnarray}
\partial_Sf(\omega)&=& \left(f(\sigma^{(b,b,b)}(\omega))-f(\sigma^{(b,b,a)}(\omega))\right) \label{eqn|lbS3ImpBDer01}\\
&& +\left(f(\sigma^{(a,a,b)}(\omega))-
f(\sigma^{(a,a,a)}(\omega))\right) \label{eqn|lbS3ImpBDer02}\\
&&-\left(f(\sigma^{(b,a,b)}(\omega))-f(\sigma^{(b,a,a)}(\omega))\right) \label{eqn|lbS3ImpBDer03}\\
&&-\left(f(\sigma^{(a,b,b)}(\omega))-f(\sigma^{(a,b,a)}(\omega))
\right). \label{eqn|lbS3ImpBDer04}
\end{eqnarray}
Assume that $\sigma^{(a,a,b)}(\omega)\in \hat E_m$. Proposition \ref{thm|forComputer} (b) implies that the term \eqref{eqn|lbS3ImpBDer02} is equal to $(b-a)$. The term 
\eqref{eqn|lbS3ImpBDer01} is non-negative. The terms \eqref{eqn|lbS3ImpBDer03} and \eqref{eqn|lbS3ImpBDer04} are negative but bounded below by $-(b-a)$. Therefore, $\partial_Sf(\omega)$ is bounded below by $(b-a)-2(b-a)=-(b-a)$. This completes the proof of \eqref{eqn|lowerBoundS3ImplicationB}. 

We proved \eqref{eqn|lowerBoundS3ImplicationMain} which states that the inequality $\partial_S f(\omega)\geq -(b-a)$ is satisfied unless 
$\sigma^{(a,a,b)}(\omega)$ is an element of $E_k\cap E_l\cap \hat E_m^C$. The analogous statements hold for $\sigma^{(a,b,a)}(\omega)$ and $\sigma^{(b,a,a)}(\omega)$. Hence, the required bound \eqref{eqn|lowerBoundS3} is proved unless all of the following three inclusions are satisfied
\begin{eqnarray}
\sigma^{(a,a,b)}(\omega)&\in& E_k\cap E_l\cap \hat E_m^C, \label{eqn|bS3Inclaab}\\
\sigma^{(a,b,a)}(\omega)&\in& E_k\cap \hat E_l^C\cap E_m,\quad\mbox{and} \label{eqn|bS3Inclaba}\\
\sigma^{(b,a,a)}(\omega)&\in& \hat E_k^C\cap E_l\cap E_m. \label{eqn|bS3Inclbaa}
\end{eqnarray}
Hence, it suffices to prove $\partial_Sf\geq -(b-a)$ under the conditions \eqref{eqn|bS3Inclaab}, \eqref{eqn|bS3Inclaba}, and \eqref{eqn|bS3Inclbaa}.
Due to $\sigma^{(a,a,b)}(\omega)\in \hat E_m^C$, there exists a geodesic on $\sigma^{(a,a,b)}(\omega)$ that does not contain the edge $m$. Let us fix one such geodesic and denote it by $\gamma_{kl}$. Since $\sigma^{(a,a,b)}(\omega)\in E_k\cap E_l$, the geodesic $\gamma_{kl}$ must contain both edges $k$ and $l$. We define the geodesics $\gamma_{lm}$ and $\gamma_{km}$ in analogous ways. Once the paths $\gamma_{kl}$, $\gamma_{lm}$, and $\gamma_{km}$ are fixed, we define the relation $\prec$ on $\{k,l,m\}$. We will write $k\prec l$ if on the path $\gamma_{kl}$ the edge $k$ appears before the edge $l$ when moving from the source to the sink. There are two cases: 
\begin{itemize}
\item Case 1: The relation $\prec$ does not have a minimum;
\item Case 2: The relation $\prec$ has a minimum.
\end{itemize}

\noindent {\bf Case 1.} This case is easier to consider. We may assume that $k\prec l$, $l\prec m$, and $m\prec k$. 

 \begin{figure}[t]
 \centering
\begin{tikzpicture}
\draw[thick] (1,3.5) -- (6,3.5) -- (4.6,2) -- (9,2);
\draw[thick,dashed] (0,2) -- (5,2) -- (5.6,0.5) -- (10,0.5); 
\draw[thick,dotted] (1,0.5) -- (6,0.5) -- (5.6,3.5) -- (10,3.5); 
\node at (5.6,3.5) {$\bullet$};
\node at (6,3.5) {$\bullet$};
\node[above] at (5.8,3.5) {$m$}; 
\node at (6,0.5) {$\bullet$};
\node at (5.6,0.5) {$\bullet$};
\node[below] at (5.8,0.5) {$l$}; 
\node at (4.6,2) {$\bullet$};
\node at (5,2) {$\bullet$};
\node[below] at (4.8,2) {$k$}; 
\node[above] at (3,3.5) {$L_m$};
\node[above] at (8,3.5) {$R_m$};
\node[above] at (3,2) {$L_k$};
\node[above] at (8,2) {$R_k$};
\node[below] at (3,0.5) {$L_l$};
\node[below] at (8,0.5) {$R_l$};
\node[above] at (5.0,2.60) {$e_{km}$};
\node[below] at (5.1,1.30) {$e_{kl}$};
\node[right] at (5.8,1.5) {$e_{lm}$};
\end{tikzpicture}
\caption{Case in which $\prec$ has no minimum.}
\label{fig|LowerBound3NoMin}
\end{figure}

As we mentioned earlier, the Proposition \ref{thm|lowerBoundS3DSwitch} allowed us to assume that the directions of the flow over the edges $k$, $l$, and $m$ is the same on the environments 
$\sigma^{(a,a,b)}(\omega)$, $\sigma^{(a,b,a)}(\omega)$, and $\sigma^{(b,a,a)}(\omega)$.  

Let us denote by $e_{kl}$ the total passage time between the edges $k$ and $l$ on the geodesic $\gamma_{kl}$. We define $e_{km}$ and $e_{lm}$ in analogous way. 
Let us denote by $L_k$ the total passage time on the geodesic $\gamma_{kl}$ before the edge $k$. Let $R_l$ be the total passage time on the geodesic $\gamma_{kl}$ after the edge $l$. 
The numbers $L_m$, $L_l$, $R_k$, and $R_m$ are defined in similar ways. Let us emphasize that none of the previously defined passage times includes the edges $k$, $m$, and $l$. 

Since the paths $\gamma_{kl}$, $\gamma_{lm}$, and $\gamma_{km}$ are fixed, the quantities $e_{\cdot\cdot}$, $R_{\cdot}$, and $L_{\cdot}$ that we defined 
above are the same on the environments $\sigma^{\overrightarrow{\theta}}(\omega)$ for all eight choices $\overrightarrow \theta \in\{a,b\}^3$.
We will now prove the following identities
\begin{eqnarray}
f(\sigma^{(a,a,b)}(\omega))&=&L_k+e_{kl}+R_l+2a,\label{eqn|lbS3Case101}\\
f(\sigma^{(a,b,a)}(\omega))&=&L_m+e_{km}+R_k+2a,\label{eqn|lbS3Case102}\\
f(\sigma^{(b,a,a)}(\omega))&\geq& f(\sigma^{(a,a,a)}(\omega)),\label{eqn|lbS3Case103}\\
f(\sigma^{(b,b,b)}(\omega))&\geq& f(\sigma^{(b,b,a)}(\omega)),\label{eqn|lbS3Case104}\\
f(\sigma^{(b,a,b)}(\omega))&\leq& L_l+a+R_l,\label{eqn|lbS3Case105}\\
f(\sigma^{(a,b,b)}(\omega))&\leq& L_k+a+R_k.\label{eqn|lbS3Case106}
\end{eqnarray}
The equalities \eqref{eqn|lbS3Case101} and \eqref{eqn|lbS3Case102} are due to the definitions of $\gamma_{kl}$ and $\gamma_{km}$. The inequalities \eqref{eqn|lbS3Case103} and 
\eqref{eqn|lbS3Case104} are the consequences of monotonicity.

Let us prove \eqref{eqn|lbS3Case105}. On the environment $\sigma^{(b,a,b)}(\omega)$, we
will construct a path $\delta$ such that 
\begin{eqnarray*}T(\delta, \sigma^{(b,a,b)}(\omega)) &=& L_l+a+R_l.
\end{eqnarray*}
Let us identify the section of the path $\gamma_{lm}$ before the edge $l$ and call it $\delta_1$. 
It has the passage time $L_l$. Let us consider the section of the path $\gamma_{kl}$ after the edge $l$. 
This section will be called $\delta_2$. Its passage time is $R_l$. 
Now we define $\delta=\delta_1\cup\{l\}\cup \delta_2$. This path $\delta$ connects the source and the sink. Its passage time is $L_l+a+R_l$, hence \eqref{eqn|lbS3Case105} must hold. 

The inequality \eqref{eqn|lbS3Case106} is proved in a similar way.

From \eqref{eqn|lbS3Case101}--\eqref{eqn|lbS3Case106} we obtain 
\begin{eqnarray}
\partial_Sf(\omega)&\geq& (L_k+e_{kl}+R_l+2a) + (L_m+e_{km}+R_k+2a) \nonumber\\
&&-(L_l+a+R_l)-(L_k+a+R_k) \nonumber \\
&=& e_{kl}   + L_m+e_{km}+ 2a 
-L_l. \label{eqn|lbS3Case107}
\end{eqnarray}
Let us consider the geodesic $\gamma_{lm}$ on the environment $\sigma^{(b,a,a)}(\omega)$. Let us consider $\hat \xi$ 
that consists of the union of the left part of $\gamma_{km}$ and the right part of $\gamma_{lm}$. 
The path $\xi=\hat\xi\cup\{m\}$ is a path between the source and the sink whose passage 
time is $L_m+a+R_m$.  Hence, we found a path between the source and the 
sink whose passage time is equal to $L_m+a+R_m$. Since $\gamma_{lm}$ is a geodesic with passage time 
$L_l+e_{lm}+R_m+2a$, we obtain
$L_m+a+R_m \geq L_l+e_{lm}+R_m+2a$. The last inequality is equivalent to $L_m-L_l\geq e_{lm}+a$. The inequality \eqref{eqn|lbS3Case107} turns into 
$\partial_Sf(\omega)\geq e_{kl}+e_{km}+e_{lm}+3a\geq 3a>-(b-a)$. We have completed the proof in Case 1.

\noindent{\bf Case 2.} We may assume that $k$ is the minimum, i.e. $k\prec l$ and $k\prec m$. Without loss of generality, we may assume that $l\prec m$. 

 \begin{figure}[t]
 \centering
\begin{tikzpicture}
\draw[thick,dotted] (0.5,2.5) -- (3,2.5) -- (7,2.5) -- (9.5,2.5);
\draw[thick,dashed] (0.5,0.5) -- (5.2,0.5) -- (7,2.5) -- (7.4,2.5) -- (7.5,2.4) -- (9.5,2.4);
\draw[thick] (0.5,2.4) -- (2.5,2.4) --(2.6,2.5) --(3,2.5) -- (4.8,0.5) -- (9.5,0.5);
\draw[thick] (7.0,2.5) -- (7.4,2.5);
\draw[thick] (4.8,0.5) -- (5.2,0.5);
\node at (3,2.5) {$\bullet$};
\node at (2.6,2.5) {$\bullet$};
\node[above] at (2.8,2.5) {$k$};
\node at (7,2.5) {$\bullet$};
\node at (7.4,2.5) {$\bullet$};
\node[above] at (7.2,2.5) {$m$};
\node at (4.8,0.5) {$\bullet$};
\node at (5.2,0.5) {$\bullet$};
\node[below] at (5,0.5) {$l$};
\node[above] at (1.5,2.5) {$L_k$};
\node[above] at (8.5,2.5) {$R_m$};
\node[above] at (5,2.5) {$e_{km}$};
\node[below] at (3.8,1.5) {$e_{kl}$};
\node[below] at (6.2,1.5) {$e_{lm}$};
\node[below] at (2.5,0.5) {$L_l$};
\node[below] at (7.5,0.5) {$R_l$};
\end{tikzpicture}
\caption{Case in which $k\prec l\prec m$.}
\label{fig|lowerBound3TriangleNoAlLRA}
\end{figure}

 Figure \ref{fig|lowerBound3TriangleNoAlLRA} corresponds to the situation in which $k\prec l\prec m$. 
 
The sections of the paths $\gamma_{kl}$ and $\gamma_{km}$ before the edge $k$ must have equal passage times. Let us denote by $L_k$ the common passage time of these sections. We may modify one of the paths $\gamma_{kl}$ and $\gamma_{km}$ in such a way that the sections before $k$ actually coincide. In a similar way, the passage times after the edge $m$ on the paths $\gamma_{km}$ and $\gamma_{lm}$ are equal. We will denote these passage times by $R_m$. We define $L_l$ as the passage time over the path $\gamma_{lm}$ before the edge $l$; $R_l$ the passage time over $\gamma_{kl}$ after $l$. We define $e_{kl}$, $e_{lm}$, and $e_{km}$ as passage times over the open intervals $(k,l)$, $(l,m)$, and $(k,m)$ on the paths 
$\gamma_{kl}$, $\gamma_{lm}$, and $\gamma_{km}$, respectively. Let us define the real number $\theta$ in such a way that $e_{km} =e_{kl}+e_{lm}+a+\theta$, i.e.
\begin{eqnarray} \theta &=&
e_{km} -e_{kl}-e_{lm}-a
 .\label{eqn|lbS3Case1Theta}\end{eqnarray} 
We don't know whether $\theta$ is positive or negative. However, we will
be able to prove that $\theta$ satisfies the inequality 
\begin{eqnarray}\theta\leq b-a,\label{eqn|InequalityTheta}\end{eqnarray}
which is sufficient for our needs. 
On the environment $\sigma^{(a,b,a)}(\omega)$, the minimal passage time is over the path $\gamma_{km}$. This passage time is $L_k+e_{km}+R_m+2a$. 
Consider the path in which the segment $(k,m)$ is replaced with $(k,l)\cup \{l\}\cup (l,m)$ (i.e., the section of $\gamma_{kl}$ between $k$ and $l$, the edge $l$, and 
the section of $\gamma_{lm}$ between $l$ and $m$). The passage time of this path is 
$L_k+e_{kl}+e_{lm}+R_m+2a+b$. The latter number is larger than the former, hence 
\[L_k+e_{kl}+e_{lm}+R_m+2a+b>L_k+e_{km}+R_m+2a,\]
which, together with \eqref{eqn|lbS3Case1Theta}, implies \eqref{eqn|InequalityTheta}.

Let us define the real numbers $\theta_L$ and $\theta_R$ with the following two identities
\begin{eqnarray}
L_l&=&L_k+e_{kl}+a+\theta_L,\label{eqn|lbS3Case1ThetaL}\\
R_l&=&R_m+e_{lm}+a+\theta_R.\label{eqn|lbS3Case1ThetaR}
\end{eqnarray}
The numbers $\theta_L$ and $\theta_R$ must belong to the interval $(0,b-a]$. Let us prove that $\theta_L\in(0,b-a]$. We need to prove $\theta_L>0$ and $\theta_L\leq b-a$. 

In order to prove
$\theta_L>0$, we start by observing that on $\sigma^{(a,a,b)}(\omega)$, every geodesic must go through $k$. The path $\gamma_{kl}$ is a geodesic. 
Let us take the section of this geodesic before the edge $l$ and replace it with the corresponding section of $\gamma_{lm}$.
We obtain a path that is not a geodesic because it omits $k$, while \eqref{eqn|bS3Inclaab} claims that every geodesic on $\sigma^{(a,a,b)}(\omega)$ must contain $k$. The change of passage time must satisfy \begin{eqnarray}0&<&L_l-L_k-a-e_{kl}= \theta_L. \label{eqn|equivalentIneqThetaL}\end{eqnarray}

Let us now prove that $\theta_L\leq b-a$. Consider the environment $\sigma^{(b,a,a)}(\omega)$. The path $\gamma_{lm}$ is a geodesic. Let us denote by $\gamma_{kl}^-$ the section of $\gamma_{kl}$ before the edge $l$. Let $\gamma_{lm}^+$ be the section of $\gamma_{lm}$ after the edge $l$. The path $\gamma_{kl}^-\cup\{l\}\cup \gamma_{lm}^+$ is a path from the source to the sink that has passage time $L_k+e_{kl}+e_{lm}+R_m+2a+b$. 
The passage time over the geodesic $\gamma_{lm}$ is 
$L_l+e_{lm}+R_m+2a$. From 
\begin{eqnarray*}L_l+e_{lm}+R_m+2a&\leq &L_k+e_{kl}+e_{lm}+R_m+2a+b,
\end{eqnarray*}
we obtain  $L_l  \leq L_k+e_{kl} +  b$. The last inequality and  \eqref{eqn|lbS3Case1ThetaL} imply $\theta_L\leq (b-a)$.
 
In an analogous way we prove that $\theta_R\in (0,b-a]$. 

We can now update our picture. Figure \ref{fig|lowerBound3TriangleNoAlLRAUpdated} shows the geodesics in which the labels $L_l$, $e_{km}$, and $R_L$ are replaced with their equivalent values involving $\theta$, $\theta_L$, and $\theta_R$.

 \begin{figure}[t]
 \centering
\begin{tikzpicture}
\draw[thick,dotted] (0.5,2.5) -- (3,2.5) -- (7,2.5) -- (9.5,2.5);
\draw[thick,dashed] (0.5,0.5) -- (5.2,0.5) -- (7,2.5) -- (7.4,2.5) -- (7.5,2.4) -- (9.5,2.4);
\draw[thick] (0.5,2.4) -- (2.5,2.4) --(2.6,2.5) --(3,2.5) -- (4.8,0.5) -- (9.5,0.5);
\draw[thick] (7.0,2.5) -- (7.4,2.5);
\draw[thick] (4.8,0.5) -- (5.2,0.5);
\node at (3,2.5) {$\bullet$};
\node at (2.6,2.5) {$\bullet$};
\node[above] at (2.8,2.5) {$k$};
\node at (7,2.5) {$\bullet$};
\node at (7.4,2.5) {$\bullet$};
\node[above] at (7.2,2.5) {$m$};
\node at (4.8,0.5) {$\bullet$};
\node at (5.2,0.5) {$\bullet$};
\node[below] at (5,0.5) {$l$};
\node[above] at (1.5,2.5) {$L_k$};
\node[above] at (8.5,2.5) {$R_m$};
\node[above] at (5,2.5) {$e_{kl}+e_{lm}+a+\theta$};
\node[below] at (3.8,1.5) {$e_{kl}$};
\node[below] at (6.2,1.5) {$e_{lm}$};
\node[below] at (2.5,0.5) {$L_k+e_{kl}+a+\theta_L$};
\node[below] at (7.5,0.5) {$R_m+e_{lm}+a+\theta_R$};
\end{tikzpicture}
\caption{The labels $L_l$, $e_{km}$, and $R_L$ are replaced with their equivalent values involving $\theta$, $\theta_L$, and $\theta_R$.}
\label{fig|lowerBound3TriangleNoAlLRAUpdated}
\end{figure}

 Since $\gamma_{kl}$, $\gamma_{lm}$, and $\gamma_{km}$ are geodesics on $\sigma^{(a,a,b)}(\omega)$, $\sigma^{(b,a,a)}(\omega)$, and $\sigma^{(a,b,a)}(\omega)$, respectively, we obtain 
\begin{eqnarray}
f(\sigma^{(a,a,b)}(\omega))&=&L_k+R_m+e_{kl}+e_{lm}+3a+\theta_R,\label{eqn|lbS3f01}\\
f(\sigma^{(a,b,a)}(\omega))&=&L_k+R_m+e_{kl}+e_{lm}+3a+\theta,\label{eqn|lbS3f02}\\
f(\sigma^{(b,a,a)}(\omega))&=&L_k+R_m+e_{kl}+e_{lm}+3a+\theta_L.\label{eqn|lbS3f03}
\end{eqnarray}
Due to monotonicity we have 
\begin{eqnarray}
f(\sigma^{(b,b,b)}(\omega))&\geq& f(\sigma^{(a,b,b)}(\omega)). \label{eqn|lbS3f04}
\end{eqnarray}
Let us consider the environment $\sigma^{(b,a,b)}(\omega)$. We will construct a path $\zeta$ that has a low passage time and that will provide a useful upper bound for $f(\sigma^{(b,a,b)}(\omega))$.   
Let $\gamma_{lm}^-$ be the section of $\gamma_{lm}$ before $l$. Let $\gamma_{kl}^+$ be the section of $\gamma_{kl}$ after $l$. 
Define $\zeta=\gamma_{lm}^-\cup \{l\}\cup \gamma_{kl}^+$. The passage time over $\zeta$ is equal to $L_k+$ $e_{kl}+$ $\theta_L+$ $R_m+$ $e_{lm}+$ $\theta_R+$ $3a$. 
The passage time over $\zeta$ is greater than or equal to the minimal passage time, hence
\begin{eqnarray}
f(\sigma^{(b,a,b)}(\omega))&\leq&L_k+R_m+e_{kl}+e_{lm}+3a+\theta_L+\theta_R. \label{eqn|lbS3f05}
\end{eqnarray}
The minimal passage time $f(\sigma^{(b,b,a)}(\omega))$ is smaller than or equal to the minimum 
of the passage times over the paths $\gamma_{lm}$ and $\gamma_{km}$, hence
\begin{eqnarray}
f(\sigma^{(b,b,a)}(\omega))&\leq& L_k+R_m+e_{kl} +e_{lm}+2a+b+\min\{\theta_L,\theta \}.\label{eqn|lbS3f06} 
\end{eqnarray}
Finally, let us consider the environment $\sigma^{(a,a,a)}(\omega)$. By considering the passage time over the path $\gamma_{km}$, we obtain 
\begin{eqnarray}
f(\sigma^{(a,a,a)}(\omega))&\leq& L_k+R_m+e_{kl}+e_{lm}+3a+\theta.\label{eqn|lbS3f0701}
\end{eqnarray}
Let us construct an alternative path $\gamma'$ on the environment $\sigma^{(a,a,a)}(\omega)$. We replace the section $(k,m)$ with the sections $(k,l)$ and $(l,m)$ of the paths $\gamma_{kl}$ and $\gamma_{lm}$. We must also add the edge $l$ in order for $\gamma'$ to connect source to the sink. The passage time over $\gamma'$ is $L_k+R_m+e_{kl}+e_{lm}+3a$, hence 
\begin{eqnarray}
f(\sigma^{(a,a,a)}(\omega))&\leq& L_k+R_m+e_{kl}+e_{lm}+3a.\label{eqn|lbS3f0702}
\end{eqnarray}
The inequalities \eqref{eqn|lbS3f0701} and \eqref{eqn|lbS3f0702} imply 
\begin{eqnarray}
f(\sigma^{(a,a,a)}(\omega))&\leq& L_k+R_m+e_{kl}+e_{lm}+3a+\min\{\theta,0\}. \label{eqn|lbS3f07}
\end{eqnarray}
We now use \eqref{eqn|lbS3f01}--\eqref{eqn|lbS3f06} and \eqref{eqn|lbS3f07} to find the lower bound on $\partial_Sf(\omega)$. Observe that $L_k+R_m+e_{kl}+e_{lm}$ appears equally many times with sign $+$ as with sign $-$. We can ignore these terms as they cancel out in the evaluation of $\partial_Sf(\omega)$. Hence, 
\begin{eqnarray}
\partial_Sf(\omega)&\geq&   a    -b+\theta -\min\{\theta_L,\theta\} -\min\{\theta,0\}. \label{eqn|lbS3AlmostDone}
\end{eqnarray}
Define
$F(\theta,\theta_L)=\theta- \min\{\theta_L,\theta\}-\min\{\theta,0\}$. It suffices to prove that $F(\theta,\theta_L)\geq 0$. 
There are two cases: $\theta\geq 0$ and $\theta<0$. If $\theta\geq 0$, then $\min\{\theta,0\}=0$ and $F(\theta,\theta_L)=\theta-\min\{\theta_L,\theta\}\geq 0$. 
If $\theta <0$, then from $\theta_L>0$ we have $\min\{\theta,\theta_L\}=\theta$, and  
$F(\theta,\theta_L)=\theta-\theta-\theta=-\theta>0$. This completes the proof of Case 2, which was the only remaining case that we needed to consider.
\end{proof}

\begin{theorem} The first four elements of the sequence $\left(\mathcal U_n\right)$ and the first three elements of the sequence $\left(\mathcal L_n\right)$ are 
\begin{eqnarray}
\left(\mathcal U_1,\mathcal U_2, \mathcal U_3, \mathcal U_4\right)&=&(1,1,1,2);
\label{eqn|UFirstElements}\\
\left(\mathcal L_1,\mathcal L_2, \mathcal L_3,\mathcal L_4\right)&=&(0,-1,-1,-2).
\label{eqn|LFirstElements}
\end{eqnarray} \end{theorem}
\begin{proof} The inequalities $\mathcal U_1\leq 1$ and $\mathcal L_1\geq 0$ follow from \eqref{eqn|firstDerivativeBound}. The bound $\mathcal U_1\geq 1$ 
can be proved by constructing an environment $\omega$ for which there is an edge $k$ such that $\partial_kf(\omega)=b-a$. This is easy to do: Let $\omega^b$ be the environment that assigns $b$ to every edge. The equality $\partial_jf(\omega^b)=b-a$ holds for every edge $j$ on the shortest path between the source and the sink. The bound $\mathcal L_1\leq 0$ 
is equally easy to prove -- the environment $\omega^a$ that assigns $a$ to every edge satisfies $\partial_kf(\omega^a)=0$ for every $k$ that is not on the shortest path.  

The inequalities $\mathcal U_2\leq 1$ and $\mathcal L_2\geq -1$ follow from \eqref{eqn|secondDerivativeBound}. Proposition \ref{thm|badSets} implies $\mathcal L_2\leq -1$ and $\mathcal U_2\geq 1$. 

In addition, Proposition \ref{thm|badSets} also implies $\mathcal U_3\geq 1$, $\mathcal L_3\leq -1$,  $\mathcal U_4\geq 2$, $\mathcal L_4\leq -2$. 
Theorem \ref{thm|upperBoundS3} implies $\mathcal U_3\leq 1$ and Theorem \ref{thm|lowerBoundS3} implies $\mathcal L_3\geq -1$. Therefore, 
$\mathcal U_3=1$ and $\mathcal L_3=-1$. 
The first inequality in \eqref{eqn|gBoundFibonacci} implies that $
\mathcal U_4\leq \mathcal U_3-\mathcal L_3=2$.  The second inequality in \eqref{eqn|gBoundFibonacci} implies that $
\mathcal L_4\geq \mathcal L_3-\mathcal U_3=-2$. 
\end{proof}


 
\section{Fourier-level contributions to the variance} \label{sec|HigherOrderInfluential}


\subsection{Influential and essential sets}\label{subs|DefinitionsHigherOrderInfluentialEssential}
The definitions of essential and influential edges can be generalized to essential and influential sets. These notions were already introduced in \cite{FirstPaper2025}. We briefly recall the definitions and their main properties, since we will need them in greater detail here. 

On an environment $\omega$, a set of edges $S$ is called \emph{essential} if every geodesic in $\omega$ passes through all elements of $S$. We denote this event by $E_S$. As noted in \cite{FirstPaper2025}, we have  
\[
E_S \;=\; \bigcap_{i\in S} E_i .
\]

We say that, on the environment $\omega$, the set of edges $S$ is \emph{influential} if $\partial_S f(\omega)\neq 0$. The event that $S$ is influential will be denoted by $A_S$. As shown in \cite{FirstPaper2025}, in general we do not have $E_S \subseteq A_S$, even for two-element subsets $S$. Moreover, $A_{\{i,j\}} \not\subseteq A_i \cup A_j$.

Compared to essential sets, influential sets are more complex and less elegant generalizations of their one-dimensional counterparts. Positive and negative values of $\partial_S f$ correspond to vastly different configurations of geodesics, so we will study them separately. We decompose $A_S$ as a union of $A_S^+$ and $A_S^-$, defined by
\[
A_S^+ \;=\; \{\omega : \partial_S f(\omega) > 0\}, 
\qquad\text{and}\qquad
A_S^- \;=\; \{\omega : \partial_S f(\omega) < 0\}.
\]

\begin{proposition}\label{thm|invarianceLevelSetsGen}
For every set of edges $S$, every edge $i\in S$, and every $\xi\in\{a,b\}$, 
\begin{eqnarray}
\left(\sigma_i^{\xi}\right)^{-1}(A_S^+)=A_S^+ 
,\quad 
\left(\sigma_i^{\xi}\right)^{-1}(A_S^-)=A_S^-,\quad\mbox{and} \quad
\left(\sigma_i^{\xi}\right)^{-1}(A_S)=A_S. &&
\label{eqn|invarianceLevelSetsGen}
\end{eqnarray}
\end{proposition}
\begin{proof}
Let us first provide a generalization of \eqref{eqn|propertyC}. Take any function $\varphi$ and apply \eqref{eqn|propertyC} to $\partial_{S\setminus\{i\}}\varphi$:  
\begin{eqnarray}
\left(\partial_S\varphi\right)\circ \sigma_i^{\xi}&=&\left(\partial_i\left(\partial_{S\setminus\{i\}}\varphi\right) \right)\circ
\sigma_i^{\xi}= \partial_i\left(\partial_{S\setminus\{i\}}\varphi\right) =\partial_S\varphi. \label{eqn|propertyCGen}
\end{eqnarray}
Since $A_S^+=\left(\partial_Sf\right)^{-1}(\mathbb R_+)$, we have 
\begin{eqnarray*}
\left(\sigma_i^{\xi}\right)^{-1}(A_S^+)&=&
\left(\sigma_i^{\xi}\right)^{-1}\left(\left(\partial_Sf\right)^{-1}(\mathbb R_+)\right)= \left(\left(\partial_Sf\right)\circ \sigma_i^{\xi}\right)^{-1}(\mathbb R_+)\\ &=&
\left( \partial_Sf \right)^{-1}(\mathbb R_+)=A_S^+.
\end{eqnarray*}
This completes the proof of the first equality in \eqref{eqn|invarianceLevelSetsGen}. The remaining two are analogous.
\end{proof}



\subsection{Influential sets with two elements}
We will work under the assumption that $a/(b-a)\in\mathbb N$. By 
Proposition~7 
in \cite{FirstPaper2025}, this condition is equivalent 
to the non-existence of integers $k_a$ and $k_b$ such that 
$ak_a+bk_b\in(0,b-a)$. In other words, no integer linear combination 
of $a$ and $b$ can land strictly between $0$ and $b-a$. 
Proposition~8 
 from \cite{FirstPaper2025} implies that for every $j$ we have \[E_j\subseteq A_j=\hat A_j\subseteq \hat E_j.\]

We will denote by $\overrightarrow\alpha$ the array of length $|S|$ whose all elements are equal to $a$.

\begin{proposition}\label{thm|NonZeroDerivativeImpliesGeodesic}
If $S$ is an influential set on the environment $\omega$, then every geodesic on $\sigma_S^{\overrightarrow \alpha}(\omega)$ contains at least one element of $S$.
\end{proposition}
\begin{proof} Assume the contrary, that there is a geodesic $\gamma$ that omits all elements of $S$. Then, for each $\overrightarrow \xi\in\{a,b\}^{|S|}$ we have
\begin{eqnarray*}
f\left(\sigma^{\overrightarrow{\xi}}_S(\omega)\right)&\leq& T\left(\gamma,\sigma^{\overrightarrow{\xi}}_S(\omega)\right)
=
T\left(\gamma,\sigma^{\overrightarrow{\alpha}}_S(\omega)\right)\\
&=&f\left(\sigma^{\overrightarrow{\alpha}}_S(\omega)\right),
\end{eqnarray*}
which implies that $f\left(\sigma^{\overrightarrow{\xi}}_S(\omega)\right)$ and $f\left(\sigma^{\overrightarrow{\alpha}}_S(\omega)\right)$ are equal. Therefore, $\partial_Sf(\omega)$ must be equal to $0$, which is a contradiction.
\end{proof}

\begin{proposition}\label{thm|Necessary2APlus} 
Assume that $a$ is an integer multiple of $b-a$. If $S=\{i,j\}$ and if $\omega\in A_S^+$, then every geodesic on $\sigma_i^a\sigma_j^a(\omega)$ contains at least one element of $\{i,j\}$ and there exist two geodesics $\gamma_i$ and $\gamma_j$ on $\sigma_i^a\sigma_j^a(\omega)$ such that $i\in \gamma_i\setminus \gamma_j$ and $j\in \gamma_j\setminus\gamma_i$.
\end{proposition}
\begin{proof} We will first prove that $\omega\in A_S^+$ implies the following identities 
\begin{eqnarray}
f\left(\sigma_i^a\sigma_j^a(\omega)\right)= f\left(\sigma_i^a\sigma_j^b(\omega)\right) = f\left(\sigma_i^b\sigma_j^a(\omega)\right) = f\left(\sigma_i^b\sigma_j^b(\omega)\right)-(b-a). \label{eqn|Necessary2APlus}
\end{eqnarray}
Assume that the first equality in \eqref{eqn|Necessary2APlus} is violated. Then, we have $\sigma_i^a(\omega)\in A_j$. Since $A_j=\hat A_j$, we have $\partial_jf\left(\sigma_i^a(\omega)\right)=b-a$ and 
\begin{eqnarray*} \partial_Sf(\omega)&=&\partial_jf\left(\sigma_i^b(\omega)\right)- \partial_jf\left(\sigma_i^a(\omega)\right) =\partial_jf\left(\sigma_i^b(\omega)\right)- (b-a)\leq 0,
\end{eqnarray*} which contradicts the assumption $\omega\in A_S^+$.
In an analogous way we prove the equality between the first and the third quantity in \eqref{eqn|Necessary2APlus} holds. Consequently, $\partial_jf\left(\sigma_i^a(\omega)\right)=0$ and 
\begin{eqnarray}\partial_Sf(\omega)=\partial_jf\left(\sigma_i^b(\omega)\right)-\partial_jf\left(\sigma_i^a(\omega)\right)=
\partial_jf\left(\sigma_i^b(\omega)\right). \label{eqn|Necessary2APlus02}
\end{eqnarray}
Since the last number is a derivative of first order, it belongs to $\{0,b-a\}$. Now, $\partial_Sf(\omega)>0$ and $\partial_Sf(\omega)\in\{0,b-a\}$ together imply $\partial_Sf(\omega)=b-a$. This together with \eqref{eqn|Necessary2APlus02} implies the last equality in \eqref{eqn|Necessary2APlus}.

Since \eqref{eqn|Necessary2APlus} is satisfied, we conclude that $\sigma_i^a\sigma_j^a(\omega)\in A_j^C\subseteq E_j^C$. Therefore, there exists a geodesic $\gamma_i$ on $\sigma_i^a\sigma_j^a(\omega)$ that omits $j$. The geodesic $\gamma_i$ must pass through at least one of the edges $i$ and $j$, due to Proposition \ref{thm|NonZeroDerivativeImpliesGeodesic}. Hence, $\gamma_i$ must pass through $i$. In an analogous way we prove that there is a geodesic $\gamma_j$ on $\sigma_i^a\sigma_j^a(\omega)$ that passes through $j$ and omits $i$.
\end{proof}

\begin{proposition}\label{thm|Sufficient2APlus} 
Let $S=\{i,j\}$. If $\omega$ is an environment such that every geodesic on $\sigma_i^a\sigma_j^a(\omega)$ contains at least one element of $\{i,j\}$ and if there exist two geodesics $\gamma_i$ and $\gamma_j$ on $\sigma_i^a\sigma_j^a(\omega)$ such that $i\in \gamma_i\setminus \gamma_j$ and $j\in \gamma_j\setminus\gamma_i$, 
then $\omega\in A_{S}^+$.
\end{proposition}
\begin{proof} Assume the contrary, that $\partial_Sf(\omega)\leq 0$. Since $j\not\in \gamma_i$, and $\gamma_i$ is a geodesic on $\sigma_i^a\sigma_j^a(\omega)$, we have
\begin{eqnarray*}
f\left(\sigma_i^a\sigma_j^b(\omega)\right)&\leq& T\left(\gamma_i,\sigma_i^a\sigma_j^b(\omega)\right) =
 T\left(\gamma_i,\sigma_i^a\sigma_j^a(\omega)\right) =f\left(\sigma_i^a\sigma_j^a(\omega)\right),
\end{eqnarray*}
hence $f\left(\sigma_i^a\sigma_j^b(\omega)\right)=f\left(\sigma_i^a\sigma_j^a(\omega)\right)$. 
In an analogous way we obtain $f\left(\sigma_i^b\sigma_j^a(\omega)\right)=f\left(\sigma_i^a\sigma_j^a(\omega)\right)$.
Let $\zeta$ 
be a geodesic on
$\sigma_i^b\sigma_j^b(\omega)$. 
\begin{eqnarray}
f\left(\sigma_i^b\sigma_j^b(\omega)\right)=T\left(\zeta,\sigma_i^b\sigma_j^b(\omega)\right) \geq
T\left(\zeta,\sigma_i^a\sigma_j^a(\omega)\right) \geq   f\left(\sigma_i^a\sigma_j^a(\omega)\right).
\label{eqn|Sufficient2APlus}
\end{eqnarray}
Since we assumed that $\partial_Sf(\omega)\leq 0$, the two inequalities $\geq$ in \eqref{eqn|Sufficient2APlus} must be equalities. Consequently, $\zeta$ is a geodesic on $\sigma_i^a\sigma_j^a(\omega)$. In addition,  
\[T\left(\zeta,\sigma_i^b\sigma_j^b(\omega)\right) =
T\left(\zeta,\sigma_i^a\sigma_j^a(\omega)\right)\] 
implies that $\zeta$ does not contain any of the elements of $\{i,j\}$. This contradicts our assumption that every geodesic on $\sigma_i^a\sigma_j^a(\omega)$ contains at least one element of $\{i,j\}$. 
\end{proof}

\begin{proposition}\label{thm|Necessary2AMinus} 
Assume that $a$ is an integer multiple of $b-a$. If $S=\{i,j\}$ and if $\omega\in A_S^-$, then every geodesic on
$\sigma_i^a\sigma_j^a(\omega)$
contains both $i$ and $j$ and there is a path $\bar\gamma$ from source to sink that contains neither $i$ nor $j$ 
such that \begin{eqnarray}
T\left(\bar\gamma,\omega\right)&=&f\left(\sigma_i^a\sigma_j^a(\omega)\right)+(b-a). 
\label{eqn|AMinusBarGamma}
\end{eqnarray}
\end{proposition}
\begin{proof} We will first prove that $\omega\in A_S^-$ implies 
\begin{eqnarray}f\left(\sigma_i^a\sigma_j^a(\omega)\right)+(b-a)= f\left(\sigma_i^a\sigma_j^b(\omega)\right) = f\left(\sigma_i^b\sigma_j^a(\omega)\right) = f\left(\sigma_i^b\sigma_j^b(\omega)\right). \label{eqn|Necessary2AMinus}
\end{eqnarray}
Assume that the last two quantities in \eqref{eqn|Necessary2AMinus} are different. Then, 
$\partial_jf\left(\sigma_i^b(\omega)\right)>0$ and $A_j=\hat A_j$ implies $\partial_jf\left(\sigma_i^b(\omega)\right)=b-a$. We now have 
\begin{eqnarray*}\partial_Sf(\omega)=\partial_jf\left(\sigma_i^b(\omega)\right)-\partial_jf\left(\sigma_i^a(\omega)\right)
=(b-a)-\partial_jf\left(\sigma_i^a(\omega)\right)\geq 0, \end{eqnarray*}
which contradicts our assumption $\omega\in A_S^-$. In the same way we prove the equality between the second and the fourth quantity in \eqref{eqn|Necessary2AMinus}. 
If the number $f\left(\sigma_i^a\sigma_j^a(\omega)\right)$ is equal to all of the last three quantities in \eqref{eqn|Necessary2AMinus}, then $\partial_Sf(\omega)$ would be $0$. Therefore,
$f\left(\sigma_i^a\sigma_j^a(\omega)\right)\neq f\left(\sigma_i^b\sigma_j^a(\omega)\right)$ and $\partial_if\left(\sigma_j^a(\omega)\right)\neq 0$. Hence, $\partial_if\left(\sigma_j^a(\omega)\right)=b-a$, which completes the proof of \eqref{eqn|Necessary2AMinus}.
Two consequences of \eqref{eqn|Necessary2AMinus} are $\sigma_i^a(\omega)\in A_j$ and $\sigma_j^a(\omega)\in A_i$.
Assume that there is a geodesic on $\sigma_i^a\sigma_j^a(\omega)$ that omits $i$. Then, 
$\sigma_j^a(\omega)\in E_i^C\subseteq A_i^C$, a contradiction. 

Therefore, every geodesic on $\sigma_i^a\sigma_j^a(\omega)$ passes through $i$. In an analogous way we prove that every geodesic on $\sigma_i^a\sigma_j^a(\omega)$ passes through $j$. 

It remains to prove that there is a path $\bar\gamma$ that omits both $i$ and $j$ for which \eqref{eqn|AMinusBarGamma} holds. It suffices to take any geodesic on $\sigma_i^b\sigma_j^b(\omega)$. The equality \eqref{eqn|AMinusBarGamma} follows directly from \eqref{eqn|Necessary2AMinus}. We only need to prove that every geodesic on $\sigma_i^b\sigma_j^b(\omega)$ omits both $i$ and $j$. Assume the contrary, that there is a geodesic on $\sigma_i^b\sigma_j^b(\omega)$ that passes through $i$. 
Then, $\sigma_i^b\sigma_j^b(\omega)\in \hat E_i$ and Proposition \ref{thm|forComputer} (b) implies 
$f\left(\sigma_i^b\sigma_j^b(\omega)\right)=f\left(\sigma_i^a\sigma_j^b(\omega)\right)+(b-a)$, which contradicts \eqref{eqn|Necessary2AMinus}. This completes the proof that every geodesic on $\sigma_i^b\sigma_j^b(\omega)$ omits $i$. The proof that every geodesic on $\sigma_i^b\sigma_j^b(\omega)$ omits $j$ is analogous.
\end{proof}

\begin{proposition}\label{thm|Sufficient2AMinus} 
Assume that $a$ is an integer multiple of $b-a$. If $S=\{i,j\}$ and if every geodesic on $\sigma_i^a\sigma_j^a(\omega)$ contains both $i$ and $j$ and there is a path $\bar\gamma$ from source to sink that contains neither $i$ nor $j$ 
such that \eqref{eqn|AMinusBarGamma} holds, then $\omega\in A_S^-$.
\end{proposition}
\begin{proof} Let us first prove that \begin{eqnarray}\partial_if\left(\sigma_j^a(\omega)\right)=b-a, 
\label{eqn|AMinusBarGamma01} \end{eqnarray} i.e. that $\sigma_j^a\left(\sigma_i^a(\omega)\right)\in \hat A_i$.
Assume the contrary, that $\sigma_j^a\left(\sigma_i^a(\omega)\right)\not\in \hat A_i$. From $\hat A_i=A_i$ and $E_i\subseteq A_i$, we obtain that there is a geodesic on $\sigma_i^a\sigma_j^a(\omega)$ that omits $i$. This contradicts our assumption. We have, therefore, proved that $\partial_if\left(\sigma_j^a(\omega)\right)=b-a$. In an analogous way we prove \begin{eqnarray}\partial_jf\left(\sigma_i^a(\omega)\right)=b-a. 
\label{eqn|AMinusBarGamma02} \end{eqnarray} 
Let us now prove that the path $\bar\gamma$ from source to sink that satisfies \eqref{eqn|AMinusBarGamma} is a geodesic on $\sigma_i^b\sigma_j^b(\omega)$. 
 We have  \begin{eqnarray}
f\left(\sigma_i^a\sigma_j^a(\omega)\right)+(b-a)=
T\left(\bar\gamma,\sigma_i^b\sigma_j^b(\omega)\right)\geq f\left(\sigma_i^b\sigma_j^b(\omega)\right)\geq 
f\left(\sigma_i^b\sigma_j^a(\omega)\right). \label{eqn|AMinusBarGammaInt}\end{eqnarray}
We already established the equality \eqref{eqn|AMinusBarGamma01}, hence 
both inequalities in \eqref{eqn|AMinusBarGammaInt} must be equalities.
Consequently, $\bar\gamma$ is a geodesic on $\sigma_i^b\sigma_j^b(\omega)$ and
\begin{eqnarray}f\left(\sigma_i^b\sigma_j^b(\omega)\right)=f\left(\sigma_i^a\sigma_j^a(\omega)\right)+b-a. 
\label{eqn|AMinusBarGamma03} \end{eqnarray} 
The equalities \eqref{eqn|AMinusBarGamma01}, \eqref{eqn|AMinusBarGamma02}, and \eqref{eqn|AMinusBarGamma03} imply that $\partial_Sf(\omega)=-(b-a)$.
\end{proof}

\begin{proposition} Assume that $a/(b-a)\in\mathbb N$. Consider the first passage percolation time  $f^{\tau}$ on the torus in $d$ dimensions, and define the events $E_i$, $E_j$, $A_i$, and $A_j$ with respect to the random variable $f^{\tau}$. Let $\zeta>0$ be a fixed real number. Then, 
there exists a constant $C>0$ such that
\begin{eqnarray}
\sum_{i,j} \left(\mathbb P\left(E_i\cap E_j\right)\right)^{1+\zeta}&\leq& C\cdot n^{2+\zeta-\zeta d}, \label{eqn|L2D3}
\end{eqnarray}
where the summation is over all pairs of edges $(i,j)$. 
Moreover, for every $\zeta'\in(0,\zeta)$ and every integer $m$ there exists a real number $D_m>0$ that depends only on $\zeta'$, $m$, $a$, $b$, and $p$ such that whenever the dimension $d$ satisfies $d\geq D_m$, the following holds
\begin{eqnarray}
\sum_{|T|=m} \left(\mathbb P\left(\bigcap_{k\in T}E_k\right)\right)^{1+\zeta}&\leq& n^{-\zeta'd}. \label{eqn|L2Dm}
\end{eqnarray}
\end{proposition}
\begin{proof} 
Recall that Proposition 16 
 from \cite{FirstPaper2025} implies that for every edge $i$,
\begin{eqnarray}
\mathbb P\left(\partial_if^{\tau}\neq 0\right)&\leq& \frac{b}{apn^{d-1}}. \label{eqn|L2D3EasierBound}
\end{eqnarray}
For shorter notation, we will write $\theta=\frac{b}{ap}$.
Using \eqref{eqn|L2D3EasierBound} we obtain
\begin{align}
\sum_{i,j}\left(\mathbb P(E_i\cap E_j)\right)^{1+\zeta}&\leq\sum_{i,j}\mathbb P(E_i)^{\zeta}\cdot \mathbb P(E_i\cap E_j)  =\sum_{i,j}\mathbb P(E_i)^{1+\zeta}\cdot \mathbb P\left(\left. E_j\right|E_i\right) \nonumber \\ 
&\leq  \frac{\theta^{1+\zeta}}{n^{(\zeta +1)(d-1)}}\cdot
\sum_i\sum_j\mathbb E\left[\left.1_{E_j}\right|E_i\right]\nonumber \\
&= \frac{\theta^{1+\zeta}}{n^{(\zeta +1)(d-1)}}\cdot \sum_i\mathbb E\left[\left.\left(\sum_j1_{E_j}\right)\right|E_i\right]. 
\label{eqn|L2D3BeforeCounting}
\end{align}
Observe that no matter what $i$ is, the summation $\sum_j1_{E_j}$ is at most equal to the total number of edges that a geodesic could possibly have. This number is bounded by $\frac ban$. Therefore, \eqref{eqn|L2D3BeforeCounting} now implies
\begin{eqnarray}\sum_{i,j}\left(\mathbb P(E_i\cap E_j)\right)^{1+\zeta}&\leq&
\frac{\theta^{1+\zeta}\cdot\frac ba\cdot n}{n^{(\zeta +1)(d-1)}}\cdot \sum_i1\leq C\frac{n^{d+1}}{n^{(\zeta +1)(d-1)}}, \label{eqn|L2D3Almost}
\end{eqnarray}
where $C=\theta^{1+\zeta}\cdot \frac ba\cdot D$, where $D$ is the integer such that $Dn^d$ is the total number of edges in the grid. The inequality \eqref{eqn|L2D3Almost} implies 
\eqref{eqn|L2D3}.

We now use induction on $m$ to prove \eqref{eqn|L2Dm}. If $m=2$, we only need to prove that 
$ Cn^{2+\zeta-\zeta d}<n^{-\zeta'd}$ for sufficiently large $d$. After taking logarithms, we obtain that it is sufficient to prove 
\begin{eqnarray}\log(C)+(2+\zeta-(\zeta-\zeta') d)\log n< 0. \label{eqn|L2Dmc2Log}
\end{eqnarray}
For sufficiently large $d$, the number $2+\zeta-(\zeta-\zeta')d$ is negative, hence the left-hand side of \eqref{eqn|L2Dmc2Log} increases if $n$ is replaced by $2$. It suffices to show that for sufficiently large $d$, the function $\psi(d)$ is negative, where 
\begin{eqnarray*}
\psi(d)&=&\log(C)+(2+\zeta-(\zeta-\zeta') d)\log 2.
\end{eqnarray*}
This is obvious because the function $\psi(d)$ is linear in $d$ and the coefficient in front of $d$ is negative. 
Let us now assume that the statement holds for fixed $m\geq 2$. 
More precisely, assume that $\zeta'\in(0,\zeta)$ is fixed and let $\zeta''=\frac12(\zeta+\zeta')$. Our induction hypothesis implies that there exists $D_m$ such that for $d\geq D_m$, the following holds 
\begin{eqnarray}
\sum_{|T|=m} \left(\mathbb P\left(\bigcap_{k\in T}E_k\right)\right)^{1+\zeta}&\leq& n^{-\zeta''d}.
\label{eqn|L2DmIH}
\end{eqnarray}
We will prove that there exists $D_{m+1}$ such that 
for $d\geq D_{m+1}$ we have
\begin{eqnarray}
\sum_{|S|=m+1} \left(\mathbb P\left(\bigcap_{k\in S}E_k\right)\right)^{1+\zeta}&\leq& n^{-\zeta'd}. \label{eqn|L2Dmp1}
\end{eqnarray}
The left-hand side of \eqref{eqn|L2Dmp1} can be written as 
\begin{align*}
\sum_{|S|=m+1} \left(\mathbb P\left(\bigcap_{k\in S}E_k\right)\right)^{1+\zeta}&=
\sum_{|T|=m}\sum_i \left(\mathbb P\left(E_i\cap \bigcap_{k\in T}E_k\right)\right)^{1+\zeta}.\end{align*}
Since the intersection of two sets is a subset of each, we derive
\begin{align*}
\sum_{|S|=m+1} \left(\mathbb P\left(\bigcap_{k\in S}E_k\right)\right)^{1+\zeta}&\leq \sum_{|T|=m}\sum_i \left(\mathbb P\left(\bigcap_{k\in T}E_k\right)\right)^{\zeta}\cdot \mathbb P\left(E_i\cap \bigcap_{k\in T}E_k\right).\end{align*}
We can express the last probability as a conditional expectation and take the common factor outside the summation in $i$.  
\begin{align}
\sum_{|S|=m+1} \left(\mathbb P\left(\bigcap_{k\in S}E_k\right)\right)^{1+\zeta}\leq
\sum_{|T|=m}\left(\mathbb P\left(\bigcap_{k\in T}E_k\right)\right)^{1+\zeta}\cdot \sum_i \mathbb E\left( 1_{E_i}\left| \bigcap_{k\in T}E_k\right.\right).   
\nonumber
\end{align}
We now use the linearity of conditional expectation and the fact that $\sum_i1_{E_i}$ is bounded above by the total number of edges that a geodesic could have. As before, we use that every geodesic can have at most $\frac ba n$ edges. Hence, the last inequality  
becomes
\begin{eqnarray}
\sum_{|S|=m+1} \left(\mathbb P\left(\bigcap_{k\in S}E_k\right)\right)^{1+\zeta}&\leq& \frac ba\cdot n
\sum_{|T|=m}\left(\mathbb P\left(\bigcap_{k\in T}E_k\right)\right)^{1+\zeta}.
\label{eqn|TowardsL2Dmp1}
\end{eqnarray}
 The inequality \eqref{eqn|L2DmIH} implies 
 \begin{eqnarray}
\sum_{|S|=m+1} \left(\mathbb P\left(\bigcap_{k\in S}E_k\right)\right)^{1+\zeta}&\leq& \frac ba\cdot n^{1-\zeta''d}.
\label{eqn|TowardsL2DAIH}
\end{eqnarray}
It suffices to prove that there exists $D_{m+1}$ such that $d\geq D_{m+1}$ implies 
$\frac ba\cdot n^{1-\zeta''d}\leq n^{-\zeta'd}$. After taking logarithms of both sides, the last inequality becomes equivalent to 
\begin{eqnarray}
\log\frac ba+(1-(\zeta''-\zeta')d)\log n&\leq& 0. \label{eqn|TowardsL2DAIHAfterLog}
\end{eqnarray}
Since $\zeta''-\zeta'>0$, for sufficiently large $d$, the number $1-(\zeta''-\zeta')d$ is negative and the left-hand side of 
\eqref{eqn|TowardsL2DAIHAfterLog} increases if $n$ is replaced with 2. Define the function 
\begin{eqnarray*}
\varphi(d)&=&\log\frac ba+(1-(\zeta''-\zeta')d)\log 2.
\end{eqnarray*}
The function is linear in $d$ and has negative derivative. Hence, $\lim_{d\to+\infty}\varphi(d)=-\infty$ and there exists $D_{m+1}$ such that \eqref{eqn|TowardsL2DAIHAfterLog} holds. This completes the induction step, and, consequently, the proof of \eqref{eqn|L2Dm}.
\end{proof}

\begin{proposition} \label{thm|BoundAijM} Assume that $a/(b-a)\in\mathbb N$.
Let $\zeta>0$. There is a constant $C$ that depends only on $\zeta$, $a$, $b$, and $p$ such that 
\begin{eqnarray}
\sum_{i,j}\left(\mathbb P(A_{\{i,j\}}^-)\right)^{1+\zeta}&\leq&C\cdot n^{2+\zeta-\zeta d}.
\label{eqn|BoundAijM}
\end{eqnarray}
\end{proposition}
\begin{proof}
From Proposition \ref{thm|Necessary2AMinus}, we have 
\begin{eqnarray*}
A_{\{i,j\}}^-&\subseteq& \left\{\omega: \sigma_i^a\sigma_j^a(\omega)\in E_i\cap E_j\right\}.
\end{eqnarray*}
Therefore, 
\begin{eqnarray}
\mathbb P\left(A_{\{i,j\}}^-\right)&\leq & 
\mathbb P\left(\omega: \sigma_i^a\sigma_j^a(\omega)\in E_i\cap E_j\right). 
\label{eqn|BoundAijM01}
\end{eqnarray}
The event $\{\sigma_i^a\sigma_j^a(\omega)\in E_i\cap E_j\}$ is independent from $\{\omega_i=a,\omega_j=a\}$ whose probability is $p^2$. Therefore, 
\eqref{eqn|BoundAijM01} implies 
\begin{align}
\mathbb P\left(A_{\{i,j\}}^-\right)&\leq 
\frac1{p^2}
\mathbb P\left( \sigma_i^a\sigma_j^a(\omega)\in E_i\cap E_j \right)\cdot\mathbb P\left(\omega_i=a,\omega_j=a\right) \nonumber\\
&=
 \frac1{p^2}
\mathbb P\left(\left\{\sigma_i^a\sigma_j^a(\omega)\in E_i\cap E_j\right\}\cap\left\{\omega_i=a,\omega_j=a\right\}\right) \nonumber\\ 
&=\frac1{p^2}
\mathbb P\left(\left\{ \omega\in E_i\cap E_j\right\}\cap\left\{\omega_i=a,\omega_j=a\right\}\right)\leq \frac1{p^2}\mathbb P\left(E_i\cap E_j\right).
\label{eqn|BoundAijM02}
\end{align}
 
The required inequality now follows from \eqref{eqn|BoundAijM02}
 and \eqref{eqn|L2D3}.
\end{proof}

\begin{proposition} \label{thm|BoundAijPl}
Assume that $a/(b-a)\in\mathbb N$.
Let $\zeta>0$. There is a constant $C$ that depends only on $\zeta$, $a$, $b$, and $p$ such that 
\begin{eqnarray}
\sum_{i,j}\left(\mathbb P(A_{\{i,j\}}^+)\right)^{1+\zeta}&\leq&C\cdot n^{2+\zeta-\zeta d}.
\label{eqn|BoundAijPl}
\end{eqnarray}
\end{proposition}

\begin{proof} By Proposition \ref{thm|invarianceLevelSetsGen}, the event $A_{\{i,j\}}^+$ depends only on the coordinates of $\omega$ outside $\{i,j\}$, while $\{\omega_i=x,\omega_j=y\}$ depends only on the coordinates inside $\{i,j\}$. Therefore $A_{\{i,j\}}^+$ is independent of $\{\omega_i=x,\omega_j=y\}$ for every pair $(x,y)\in\{a,b\}^2$.

\begin{eqnarray*}
\mathbb P\left(A_{\{i,j\}}^+\right)&=&\sum_{(x,y)\in\{a,b\}^2} \mathbb P\left(A_{\{i,j\}}^+\cap \{\omega_i=x,\omega_j=y\}\right) \\ &=&
\sum_{(x,y)\in\{a,b\}^2} \mathbb P\left(A_{\{i,j\}}^+\right)\cdot \mathbb P\left(\omega_i=x,\omega_j=y\right).\end{eqnarray*}
If we define $\rho=\max\left\{\frac{1-p}p,\frac{p}{1-p}\right\}$, then for every edge $i$ we have $\mathbb P(\omega_i=x)\leq \rho\mathbb P(\omega_i=a)$ and 
$\mathbb P(\omega_i=x)\leq \rho\mathbb P(\omega_i=b)$, hence
\begin{eqnarray*}
\mathbb P\left(A_{\{i,j\}}^+\right)
&\leq& \rho^2\sum_{(x,y)\in\{a,b\}^2} \mathbb P\left(A_{\{i,j\}}^+\right)\cdot \mathbb P\left(\omega_i=a,\omega_j=b\right) \\ &=& 
\rho^2\sum_{(x,y)\in\{a,b\}^2} \mathbb P\left(A_{\{i,j\}}^+,\omega_i=a,\omega_j=b\right).
\end{eqnarray*} 
Since all the four summands are equal, we obtain
\begin{eqnarray}
\mathbb P\left(A_{\{i,j\}}^+\right)
 &\leq&4 \rho^2 \mathbb P\left(A_{\{i,j\}}^+,\omega_i=a,\omega_j=b\right). \label{eqn|BeforeConditioningIAJB}
 \end{eqnarray}
The following inequality is obtained in a similar way.
 \begin{eqnarray}
\mathbb P\left(A_{\{i,j\}}^+\right)
 &\leq&4 \rho^2 \mathbb P\left(A_{\{i,j\}}^+,\omega_i=b,\omega_j=a\right). \label{eqn|BeforeConditioningIBJA}
 \end{eqnarray}
Since $A_{i,j}^+\cap \{\omega_i=a,\omega_j=b\}\subseteq E_i$, we use \eqref{eqn|BeforeConditioningIAJB} to derive
\begin{eqnarray} \mathbb P\left(A_{\{i,j\}}^+\right)^{\zeta}\leq  (4\rho^2)^{\zeta}\mathbb P(E_i)^{\zeta}. \label{eqn|AijPlRaisedToZetaOnly} \end{eqnarray} If we multiply \eqref{eqn|AijPlRaisedToZetaOnly} with \eqref{eqn|BeforeConditioningIBJA}, we obtain 
\begin{align}\nonumber
\mathbb P\left(A_{\{i,j\}}^+\right)^{1+\zeta}&\leq (4\rho^2)^{1+\zeta}\cdot \mathbb P(E_i)^{\zeta} \cdot 
\mathbb P\left(A_{\{i,j\}}^+,\omega_i=b,\omega_j=a\right)\\  
&=
(4\rho^2)^{1+\zeta}\cdot \mathbb P(E_i)^{1+\zeta} \cdot 
\mathbb P\left(\left.A_{\{i,j\}}^+,\omega_i=b,\omega_j=a\right|E_i\right).  \label{eqn|beforeEj}
\end{align}
We will now show that \eqref{eqn|Necessary2APlus} allows us to conclude  
\begin{eqnarray}A_{\{i,j\}}^+\cap\{\omega_i=b\}\cap\{\omega_j=a\} & \subseteq & E_j.
\label{eqn|consequenceOfNecessary2APlus} 
\end{eqnarray} 
Indeed, if $\omega$ belongs to the left-hand side of \eqref{eqn|consequenceOfNecessary2APlus}, then 
$\omega=\sigma_i^b\sigma_j^a(\omega)$. If we assumed the contrary, that $\sigma_i^b\sigma_j^a(\omega)\in E_j^C$, Proposition \ref{thm|forComputer} (a) would guarantee that $f(\sigma_i^b\sigma_j^b(\omega))=
f(\sigma_i^b\sigma_j^a(\omega))$, contradicting \eqref{eqn|Necessary2APlus}. We can now use
the inequality \eqref{eqn|beforeEj} and the inclusion 
\eqref{eqn|consequenceOfNecessary2APlus} to derive  
\begin{eqnarray}
\mathbb P\left(A_{\{i,j\}}^+\right)^{1+\zeta} 
&\leq& (
4\rho^2)^{1+\zeta}\cdot \mathbb P(E_i)^{1+\zeta} \cdot 
\mathbb P\left(\left.E_j\right|E_i\right).
\label{eqn|AijPlBeforeConditioning}
\end{eqnarray}
We now sum \eqref{eqn|AijPlBeforeConditioning} over all edges $j$ and use linearity of conditional expectation to obtain
\begin{eqnarray*}
\sum_j\mathbb P\left(A_{i,j}^+\right)^{1+\zeta}&\leq& (4\rho^2)^{1+\zeta}\mathbb P\left(E_i\right)^{1+\zeta}\cdot \mathbb E\left[\left.
\sum_j 1_{E_j}\right|E_i
\right].
\end{eqnarray*} 
The summation $\sum_j1_{E_j}$ is at most equal to the total number of edges that a geodesic could possibly have. Since $\sum_j1_{E_j}\leq \frac ba\cdot n$, we obtain 
\begin{eqnarray*}
\sum_j\mathbb P\left(A_{i,j}^+\right)^{1+\zeta}&\leq& (4\rho^2)^{1+\zeta}\mathbb P\left(E_i\right)^{1+\zeta}\cdot \frac ba\cdot n.
\end{eqnarray*} 
We now sum over all $Dn^d$ vertices $i$ and use \eqref{eqn|L2D3EasierBound} to obtain
\begin{eqnarray*}
\sum_i\sum_j\mathbb P\left(A_{i,j}^+\right)^{1+\zeta}&\leq& (4\rho^2)^{1+\zeta}\cdot D\cdot \frac ba\cdot\left(\frac{b}{ap}\right)^{1+\zeta} \cdot \frac{n^{d+1}}{n^{(1+\zeta)(d-1)}}.
\end{eqnarray*} 
The last inequality implies \eqref{eqn|BoundAijPl}.
\end{proof}

\begin{proof}[Proof of Theorem \ref{thm|BoundFourierLevel2Torus}]
The inequality \eqref{eqn|BoundFourierLevel1Torus} was established in the proof of Corollary 1 
 in \cite{FirstPaper2025}. It is an easy consequence of \eqref{eqn|L2D3EasierBound}, the symmetry, and the fact that there is a total of $O( n^d)$ edges in the graph. 

We need to prove \eqref{eqn|BoundFourierLevel2Torus}, that is $\Sigma_2(f^{\tau})\leq Cn^{3-d}$ for some constant $C$. The level-2 Fourier sum can be written as 
\begin{eqnarray}
\Sigma_2(f^{\tau})&\leq& (b-a)^2\cdot \sum_{|S|=2} \left(\mathbb P(\partial_Sf\neq 0)\right)^2 = (b-a)^2
\cdot 
\sum_{|S|=2} \left(\mathbb P(A_S^-)+P(A_S^+)\right)^2\nonumber \\
&\leq& 2(b-a)^2\left(\sum_{|S|=2}\left(\mathbb P(A_S^-)\right)^2 +\sum_{|S|=2}\left(\mathbb P(A_S^+)\right)^2 \right).
\label{eqn|BoundFourierLevel2Torus01}
\end{eqnarray}
We now use \eqref{eqn|BoundAijM} and \eqref{eqn|BoundAijPl} with $\zeta=1$ and \eqref{eqn|BoundFourierLevel2Torus01} to conclude that there is a constant $D$ such that
\begin{eqnarray}
\Sigma_2(f^{\tau})&\leq& 2(b-a)^2D\cdot n^{3-d}.
\end{eqnarray}
By choosing $C=2(b-a)^2D$ we obtain \eqref{eqn|BoundFourierLevel2Torus}.
\end{proof}



\subsection{Bounds on sums of probabilities of large essential and influential sets}\label{subs|GeneralM}

\begin{proposition}\label{thm|introducingAAtCost}
There exists a constant $C$ that depends only on $a$, $b$, $p$, $d$, and $m$ such that for every 
 set $T$ of $m$ edges and every edge $k\not \in T$, the following inequality holds
 \begin{eqnarray}
 \mathbb P\left(\partial_{T\cup \{k\}}f\neq 0\right)&\leq& C\cdot \mathbb P\left(\{\partial_Tf\neq 0\}\cap A_k\right). \label{eqn|introducingAAtCost}
 \end{eqnarray}
 \end{proposition}
 \begin{proof} Let us denote $S=T\cup \{k\}$. Let $\overrightarrow t$ be a vector of length $m$ whose components are the elements of $T$. 
 Let $\overrightarrow s$ be the vector of length $m+1$ whose last component is $k$ and the first $m$ components are the same as $\overrightarrow t$. Let us first prove that
\begin{eqnarray}
\left\{\partial_Sf\neq 0\right\}&\subseteq& \bigcup_{\overrightarrow \xi\in \{a,b\}^{m}} \left\{
\sigma^{\overrightarrow \xi}_{\overrightarrow t}(\omega)\in A_k\right\}. \label{eqn|UnionOverXiEk}
\end{eqnarray}
If we assume that \eqref{eqn|UnionOverXiEk}
were not true, then there exists $\omega$ such that $\partial_Sf(\omega)\neq 0$ and  
\begin{eqnarray}\partial_kf\left(\sigma^{\overrightarrow \xi}_{\overrightarrow t}(\omega)\right)=0, \label{eqn|ViolationXiEk}
\end{eqnarray}
for every $\overrightarrow\xi\in\{a,b\}^m$. We can group the terms $\sigma_k^b\sigma^{\overrightarrow \xi}_{\overrightarrow t}(\omega)$ and 
 $\sigma_k^a\sigma^{\overrightarrow \xi}_{\overrightarrow t}(\omega)$ in \eqref{eqn|LeibnitzRule} and use \eqref{eqn|ViolationXiEk} to obtain
\begin{eqnarray}
\partial_Sf(\omega)&=&\sum_{\overrightarrow\xi\in\{a,b\}^m}
(-1)^{\mathbf1_a(\xi_1)+ \cdots+\mathbf1_a(\xi_m)}\partial_kf\left(\sigma^{\overrightarrow \xi}_{\overrightarrow t}(\omega)\right)=0.
\end{eqnarray}
This contradicts the assumption $\partial_Sf(\omega) \neq 0$. We have proved \eqref{eqn|UnionOverXiEk}, which we can now use to obtain
\begin{eqnarray}
\mathbb P\left(\partial_Sf(\omega)\neq 0\right)&\leq& \sum_{\overrightarrow\xi\in\{a,b\}^m}
\mathbb P\left(\left\{\partial_Sf(\omega)\neq 0\right\}\cap \left\{
\sigma^{\overrightarrow \xi}_{\overrightarrow t}(\omega)\in A_k\right\}\right)
.
 \label{eqn|afterUnionOverXiEk}
\end{eqnarray}
Let us analyze the event $\left\{\sigma_{\overrightarrow t}^{\overrightarrow \xi}(\omega)=\omega\right\}$. It states that the coordinates of $\omega$ that correspond to the edges of $\overrightarrow t$ are assigned the values $\overrightarrow \xi$. The probability of this event satisfies 
\begin{eqnarray}\mathbb P
\left(\sigma_{\overrightarrow t}^{\overrightarrow \xi}(\omega)=\omega\right)&=& \prod_{i=1}^m p^{1_a(\xi_i)}\cdot (1-p)^{1_b(\xi_i)}\geq \left( \min\{p,1-p\}\right)^m. \label{eqn|rhoBound}
\end{eqnarray}
Let us introduce a constant $\rho=\left( \min\{p,1-p\}\right)^{-m}$. This constant depends only on $p$ and $m$.
The inequalities \eqref{eqn|afterUnionOverXiEk} and \eqref{eqn|rhoBound} imply
\begin{eqnarray}
\mathbb P\left(\partial_Sf(\omega)\neq 0\right)&\leq& \rho\sum_{\overrightarrow\xi\in\{a,b\}^m}
\mathbb P\left(\left\{\partial_Sf(\omega)\neq 0\right\}\cap \left\{
\sigma^{\overrightarrow \xi}_{\overrightarrow t}(\omega)\in A_k\right\}\right) \nonumber \\ &&\quad\quad\times  \mathbb P\left(\sigma_{\overrightarrow t}^{\overrightarrow \xi}(\omega)=\omega\right)
.
 \label{eqn|afterUnionOverXiEk2}
\end{eqnarray}
 
The event $\left\{\partial_Sf(\omega)\neq 0\right\}$ depends only on the coordinates of $\omega$ outside $\overrightarrow s$ by Proposition \ref{thm|invarianceLevelSetsGen} together with \eqref{eqn|LeibnitzRule}. The event $\left\{\sigma^{\overrightarrow \xi}_{\overrightarrow t}(\omega)\in A_k\right\}$ admits the rewrite
\begin{eqnarray*}
\left\{\sigma^{\overrightarrow \xi}_{\overrightarrow t}(\omega)\in A_k\right\} 
&=&
\left\{\sigma_k^a\circ\sigma^{\overrightarrow \xi}_{\overrightarrow t}(\omega)\in E_k\right\},
\end{eqnarray*}
which makes its dependence on the coordinates of $\omega$ outside $\overrightarrow s$ apparent: $\sigma_k^a$ overwrites the $k$-th coordinate and $\sigma^{\overrightarrow \xi}_{\overrightarrow t}$ overwrites the coordinates in $\overrightarrow t$. Therefore, the intersection of the events $\left\{\partial_Sf(\omega) \neq 0\right\}\cap \left\{
\sigma^{\overrightarrow \xi}_{\overrightarrow t}(\omega)\in A_k\right\}$ is independent from $\left\{
\sigma_{\overrightarrow t}^{\overrightarrow\xi}(\omega)=\omega
\right\}$, because the former depends only on coordinates outside $\overrightarrow s$ while the latter depends only on coordinates inside $\overrightarrow t\subseteq \overrightarrow s$.

The inequality \eqref{eqn|afterUnionOverXiEk2} implies 
\begin{eqnarray}
\mathbb P\left(\partial_Sf(\omega)\neq 0\right)&\leq& \rho \sum_{\overrightarrow\xi\in\{a,b\}^m}
\mathbb P\Big(\left\{\partial_Sf\left(\sigma_{\overrightarrow t}^{\overrightarrow\xi}(\omega)\right)\neq 0\right\} \cap \left\{
\sigma^{\overrightarrow \xi}_{\overrightarrow t}(\omega)\in A_k\right\}
\nonumber \\ &&\quad\quad\quad\quad\quad\quad \cap
\left\{\sigma_{\overrightarrow t}^{\overrightarrow\xi}(\omega)=\omega\right\}\Big) \nonumber \\
&=&
 \rho \sum_{\overrightarrow\xi\in\{a,b\}^m}
\mathbb P\Big(\left\{\partial_Sf(\omega)\neq 0\right\} \cap \left\{
\omega\in A_k\right\} \cap
\left\{\sigma_{\overrightarrow t}^{\overrightarrow\xi}(\omega)=\omega\right\}\Big).
\nonumber
\end{eqnarray}
Dropping the constraint $\left\{\sigma_{\overrightarrow t}^{\overrightarrow\xi}(\omega)=\omega\right\}$ and summing over all $2^m$ choices of $\overrightarrow\xi$ yields
\begin{eqnarray}
\mathbb P\left(\partial_Sf(\omega)\neq 0\right)&\leq& 
 \rho \sum_{\overrightarrow\xi\in\{a,b\}^m}
\mathbb P\Big(\left\{\partial_Sf(\omega)\neq 0\right\} \cap A_k\Big)\nonumber 
\\
& =&
 2^m\rho \cdot \mathbb P\Big(\left\{\partial_Sf(\omega)\neq 0\right\} \cap A_k\Big).
 \label{eqn|afterUnionOverXiEkAfterIndependence}
\end{eqnarray}

If we write $\partial_Sf(\omega)$ as the difference of derivatives with respect to $T$
\begin{eqnarray*}
\partial_Sf(\omega)&=&\partial_Tf\left(\sigma_k^b(\omega)\right)-\partial_Tf\left(\sigma_k^a(\omega)\right),
\end{eqnarray*}
it becomes clear that \begin{eqnarray} \left\{\partial_Sf(\omega)\neq 0\right\}&\subseteq&
\Big\{\partial_Tf\left(\sigma_k^b(\omega)\right)\neq 0\Big\}\cup
\Big\{\partial_Tf\left(\sigma_k^a(\omega)\right)\neq 0\Big\}. \label{eqn|SIntoTwoTs}
\end{eqnarray}
The inequalities \eqref{eqn|afterUnionOverXiEkAfterIndependence}
and \eqref{eqn|SIntoTwoTs} imply
\begin{eqnarray}
\mathbb P\left(\partial_Sf(\omega)\neq 0\right)&\leq&  
 2^m\rho\Big( \mathbb P\Big(\left\{\partial_Tf\left(\sigma_k^b(\omega)\right)\neq 0\right\} \cap A_k\Big) \nonumber \\
 &&\quad\quad
 +
 \mathbb P\Big(\left\{\partial_Tf\left(\sigma_k^a(\omega)\right)\neq 0\right\} \cap A_k\Big)\Big) \label{eqn|afterTabk2}.
 \end{eqnarray}
We can now use $A_k=(\sigma_k^b)^{-1}(A_k)$, which follows from \eqref{eqn|invarianceLevelSets}, to re-write the first term on the right-hand side of \eqref{eqn|afterTabk2} as
\begin{eqnarray}
\mathbb P\Big(\left\{\partial_Tf\left(\sigma_k^b(\omega)\right)\neq 0\right\} \cap A_k\Big) =
\mathbb P\Big(\left\{\partial_Tf\left(\sigma_k^b(\omega)\right)\neq 0\right\}  
   \cap \left(\sigma_k^b\right)^{-1}\left(A_k\right)\Big). &&\label{eqn|afterTabk2Term1}
\end{eqnarray}
The right-hand side of \eqref{eqn|afterTabk2Term1} is an intersection of two events, each of which is determined by first flipping the $k$-th bit of $\omega$ into $b$. Hence, the intersection on the right-hand side of \eqref{eqn|afterTabk2Term1} is independent of the event $\{\omega_k=b\}$, whose probability is $(1-p)$. 
Hence, \eqref{eqn|afterTabk2Term1} implies 
\begin{align}
\mathbb P\Big(\left\{\partial_Tf\left(\sigma_k^b(\omega)\right)\neq 0\right\} \cap A_k\Big)&=
\frac1{1-p}\cdot \mathbb P\Big(\left\{\partial_Tf\left(\sigma_k^b(\omega)\right)\neq 0\right\} \nonumber \\
&\quad\quad  \cap \left\{\omega:  \sigma_k^b(\omega)\in A_k\right\}\Big)\cdot \mathbb P(\omega_k=b)
\nonumber \\
&= 
\frac1{1-p}\cdot \mathbb P\Big(\left\{\partial_Tf(\omega)\neq 0\right\} \cap A_k \cap \{\omega_k=b\}\Big) \nonumber \\
&\leq 
\frac1{1-p}\cdot \mathbb P\Big(\left\{\partial_Tf(\omega)\neq 0\right\} \cap A_k\Big) 
. \label{eqn|afterTabk2Term1Bound}
\end{align}

The second term on the right-hand side of \eqref{eqn|afterTabk2} satisfies
\begin{align}
\mathbb P\Big(\left\{\partial_Tf\left(\sigma_k^a(\omega)\right)\neq 0\right\} \cap A_k\Big)\Big)&=
\mathbb P\Big(\left\{\partial_Tf\left(\sigma_k^a(\omega)\right)\neq 0\right\} \cap (\sigma_k^a)^{-1}(E_k)\Big)\Big)\nonumber \\
&\leq \frac1p
\mathbb P\Big(\left\{\partial_Tf\left(\sigma_k^a(\omega)\right)\neq 0\right\} \cap (\sigma_k^a)^{-1}(E_k)\Big)\Big) \nonumber \\
&\quad \times \mathbb P(\omega_k=a)\nonumber \\
&=\frac1p
\mathbb P\Big(\left\{\partial_Tf\left(\sigma_k^a(\omega)\right)\neq 0\right\} \cap \left\{\sigma_k^a(\omega)\in E_k\right\} \nonumber \\
&\quad\quad \cap \{\omega_k=a\}\Big)  \nonumber \\
&=
\frac1p
\mathbb P\Big(\left\{\partial_Tf(\omega )\neq 0\right\} \cap E_k  \cap \{\omega_k=a\}\Big) .
\label{eqn|afterTabk3} 
\end{align}
The right-hand side of \eqref{eqn|afterTabk3}  is smaller than or equal to $\mathbb P\left(\left\{\partial_Tf(\omega)\neq 0\right\}\cap A_k\right)$. This inequality together with \eqref{eqn|afterTabk2Term1Bound} and 
\eqref{eqn|afterTabk2} implies 
\eqref{eqn|introducingAAtCost}. 
 \end{proof}

 \begin{conjecture}\label{conj|introducingEAtCost}
There exists a constant $C$ that depends only on $a$, $b$, $p$, $d$, and $m$ such that for every 
 set $T$ of $m$ edges and every edge $k\not \in T$, the following inequality holds
 \begin{eqnarray}
 \mathbb P\left(\partial_{T\cup \{k\}}f\neq 0\right)&\leq& C\cdot \mathbb P\left(\{\partial_Tf\neq 0\}\cap E_k\right). \label{eqn|introducingEAtCost}
 \end{eqnarray}
 \end{conjecture}
 
 \begin{proposition}\label{thm|inductionStep}
Let $\zeta>0$ be a fixed positive real number. Assume that the Conjecture \ref{conj|introducingEAtCost} is true. Assume that there exists a positive real number $\alpha_m$ and a positive real number $D_m$ such that for all dimensions $d\geq D_m$, the following holds
\begin{eqnarray}
\sum_{|T|=m} \left(\mathbb P\left(\partial_Tf\neq 0\right)\right)^{1+\zeta}&\leq&  n^{-\alpha_md}, \label{eqn|boundOrderk}
\end{eqnarray}
then for every positive real number $\alpha_{m+1}\in(0,\alpha_m)$ there exists a positive real number $D_{m+1}$ such that for all dimensions $d\geq D_{m+1}$, the following holds
\begin{eqnarray}
\sum_{|S|=m+1} \left(\mathbb P\left(\partial_Sf\neq 0\right)\right)^{1+\zeta}&\leq& n^{-\alpha_{m+1}d}. \label{eqn|boundOrderkp1}
\end{eqnarray}
 \end{proposition}
 \begin{proof} We can re-write the summation as 
 \begin{align}
 \sum_{|S|=m+1}\left(\mathbb P\left(\partial_Sf\neq 0\right)\right)^{1+\zeta}&=\sum_{|T|=m}\sum_k\left(\mathbb P\left(\partial_Sf\neq 0\right)\right)^{1+\zeta} \nonumber \\
 &\leq C \sum_{|T|=m}\sum_k\left(\mathbb P\left(\left\{\partial_Tf\neq 0\right\}\cap E_k\right)\right)^{1+\zeta} \nonumber \\
 &\leq C \sum_{|T|=m}\mathbb P\left(\partial_Tf\neq 0\right)^{\zeta}\sum_k\mathbb P\left(\left\{\partial_Tf\neq 0\right\}\cap E_k\right)\nonumber 
 \\
  &\leq C \sum_{|T|=m}\mathbb P\left(\partial_Tf\neq 0\right)^{1+\zeta}\sum_k\mathbb E\left[ \left.1_{E_k}\right|\left\{\partial_Tf\neq 0\right\}\right].
  \label{eqn|transformedIntoConditionalExpectation}
 \end{align}
Each geodesic can have at most $\frac ban$ edges. 
 The last summation in \eqref{eqn|transformedIntoConditionalExpectation} is bounded above by $\frac ban$. Hence, for $d\geq D_m$, 
 \eqref{eqn|transformedIntoConditionalExpectation} implies 
 \begin{align}
  \sum_{|S|=m+1}\left(\mathbb P\left(\partial_Sf\neq 0\right)\right)^{1+\zeta}&\leq 
  C\cdot \frac{b}{a}\cdot n\sum_{|T|=m}\mathbb P\left(\partial_Tf\neq 0\right)^{1+\zeta} \nonumber \\ & \leq C\cdot\frac ba\cdot n^{1-\alpha_md}. \quad
  \label{eqn|assumptionApplied}\end{align}
  We now need to prove that there exist real numbers $\alpha_{m+1}$ and $D_{m+1}$ such that $d\geq D_{m+1}$ implies 
  \begin{eqnarray}
  C\cdot\frac ba\cdot n^{1-\alpha_md}\leq n^{-\alpha_{m+1}d}.
  \label{eqn|RequirementForAlphaD}
 \end{eqnarray}
 After taking logarithms we obtain that \eqref{eqn|RequirementForAlphaD} is equivalent to 
\begin{eqnarray}\log\frac{bC}{a} + \left(1-\left(\alpha_m-\alpha_{m+1}\right)d\right)\log n\leq 0.
  \label{eqn|RequirementForAlphaDAfterLog}
\end{eqnarray}
The left-hand side of \eqref{eqn|RequirementForAlphaDAfterLog} is bounded above by 
\[\psi(d)=\log\frac{bC}{a} + \left(1-\left(\alpha_m-\alpha_{m+1}\right)d\right)\log 2,\] that clearly satisfies 
\[\lim_{d\to+\infty}\psi(d)=-\infty.\] Thus, the required $D_{m+1}$ must exist.
 \end{proof}
 



\begin{appendix} 
 
\section{Evaluations of derivatives in extreme environments} \label{Appendix|EvaluationsInExtremeEnvironments}

\begin{proposition} If $A$ and $B$ are two integers such that $0\leq A\leq B$, then
\begin{eqnarray}
\sum_{k=A}^B(-1)^k\binom{B}{k}&=&\left\{\begin{array}{ll}(-1)^A\cdot \binom{B-1}{A-1},& A\geq 1,\\
0,& A=0.\end{array}\right.
\label{eqn|ABAlternation}
\end{eqnarray}
\end{proposition}
\begin{proof}
In the case $A=0$, the sum on the left-hand side becomes $(1-1)^B$, which is $0$. 

It remains to consider the case $A\geq 1$. 
Notice that for $B-1\geq k$ the following holds: $\binom{B}{k}=\binom{B-1}k+\binom{B-1}{k-1}$.
\begin{eqnarray*}
\sum_{k=A}^B(-1)^k\binom{B}{k}&=&
(-1)^B\binom BB+
\sum_{k=A}^{B-1}(-1)^k\binom{B}{k}\\
&=&
(-1)^B+
\sum_{k=A}^{B-1}(-1)^k\binom{B-1}{k}+\sum_{k=A}^{B-1}(-1)^k\binom{B-1}{k-1}.
\end{eqnarray*}
We substitute $k'=k-1$ in the second summation.
\begin{eqnarray*}
\sum_{k=A}^B(-1)^k\binom{B}{k} 
&=&
(-1)^B+
\sum_{k=A}^{B-1}(-1)^k\binom{B-1}{k}-\sum_{k'=A-1}^{B-2}(-1)^{k'}\binom{B-1}{k'}\\
&=&(-1)^B+(-1)^{B-1}\cdot\binom{B-1}{B-1}-(-1)^{A-1}\cdot\binom{B-1}{A-1}\\
&=&(-1)^A\cdot \binom{B-1}{A-1}.
\end{eqnarray*}
This completes the proof of \eqref{eqn|ABAlternation}.
\end{proof}

\begin{proposition}
If $n$ and $k$ are non-negative integers, and $A$ an integer such that $n+A\in[0,k+n]$, the following holds
\begin{eqnarray}
\sum_{i=\max\{0,A\}}^{\min\{k,n+A\}} \binom ki\cdot\binom{n}{n+A-i}&=&\binom{n+k}{n+A}. \label{eqn|CatsAndDogs}
\end{eqnarray}
\end{proposition}

\begin{proof}
This is a famous cats-and-dogs identity. Assume that there are $k$ cats and $n$ dogs and that we want to count the number of ways to choose a committee consisting of $n+A$ animals. One obvious way to count committees is to ignore the differences between cats and dogs. The number becomes $\binom{n+k}{n+A}$. However, we can also do a counting by doing the case-work. If we denote by $i$ the number of cats in the committee, then the number $i$ must range from $\max\{0,A\}$ to $\min\{k,n+A\}$. The number of $(n+A)$-member committees with exactly $i$ cats is $\binom ki\cdot \binom{n}{n+A-i}$. 
\end{proof}

 \begin{proof}[Proof of Theorem \ref{thm|dLanes}]
Let $L$ be the total number of edges on each of the paths $\gamma_1$ and $\gamma_2$. We know that $L=2n+2N$, but it is more elegant to express everything in terms of $L$ and avoid $n$ and $N$. 

Assume, first, that $\beta_1-\beta_2\geq m_2$. The path $\gamma_2$ is a geodesic on every environment because 
the passage time over $\gamma_1$ is larger than or equal to $La+\beta_1(b-a)$, while the passage time over $\gamma_2$ is smaller than or equal to 
$La+(m_2+\beta_2)\cdot(b-a)$. Hence, if $i_2$ of the edges from $C_2$ have the value $b$, then the shortest passage time is equal to $La+(i_2+\beta_2)(b-a)$. 
Hence, \begin{eqnarray*}
\partial_Sf(\omega)&=&(b-a)\cdot \sum_{i_1=0}^{m_1}\sum_{i_2=0}^{m_2}(-1)^{m_1-i_1+m_2-i_2}\cdot \binom{m_1}{i_1}\cdot \binom{m_2}{i_2}\cdot (i_2+\beta_2)\\
&=&(b-a)\cdot A_1 \cdot A_2,
\end{eqnarray*}
where
\begin{eqnarray*}
A_1&=&\sum_{i_1=0}^{m_1}(-1)^{m_1+i_1}\binom{m_1}{i_1},\\
A_2&=&\sum_{i_2=0}^{m_2}(-1)^{m_2+i_2}\cdot \binom{m_2}{i_2}\cdot(i_2+\beta_2).
\end{eqnarray*}
The first factor satisfies $A_1=(1-1)^{m_1}=0$, hence $\partial_Sf(\omega)=0$. The case $\beta_2-\beta_1\geq m_1$ is analogous.

We now treat the case in which both $\beta_1-\beta_2\leq m_2-1$ and $\beta_2-\beta_1\leq m_1-1$. Assume that $i_1$ of the edges from $C_1$ and $i_2$ edges from $C_2$ are assigned the value $b$. The passage time over the path $\gamma_1$ is $La+(b-a)(i_1+\beta_1)$. The passage time over $\gamma_2$ is $La+(b-a)(i_2+\beta_2)$. Therefore, the derivative satisfies 
\begin{eqnarray*}
\partial_Sf(\omega)&=&(b-a)\cdot(-1)^{m_1+m_2}\\ 
&&\times\sum_{i_1=0}^{m_1}\sum_{i_2=0}^{m_2}(-1)^{i_1+i_2}\cdot \binom{m_1}{i_1}\cdot\binom{m_2}{i_2}\cdot \min\{i_1+\beta_1,i_2+\beta_2\}.
\end{eqnarray*}
Let us express the derivative as 
\begin{align}
\partial_Sf(\omega)&=(b-a)\cdot(-1)^{m_1+m_2}\cdot(S_1+S_2), \label{eqn|S1PlusS2}
\end{align}
where the summation $S_1$ is carried over the pairs $(i_1,i_2)$ for which $i_1+\beta_1\leq i_2+\beta_2$, and the summation $S_2$ is over the pairs $(i_1,i_2)$ for which $i_1+\beta_1\geq i_2+\beta_2+1$. 

If $i_1+\beta_1\leq i_2+\beta_2$, then $i_1\leq i_2+\beta_2-\beta_1\leq m_2+\beta_2-\beta_1$. Therefore, the number $S_1$ satisfies
\begin{align}
S_1&=\sum_{i_1=0}^{\mu_1}(-1)^{i_1}\cdot (i_1+\beta_1)\cdot \binom{m_1}{i_1}\cdot \sum_{i_2=\tau_2(i_1)}^{m_2}(-1)^{i_2}\cdot\binom{m_2}{i_2}, \;\mbox{where}\quad \label{eqn|S1} \\
\mu_1&=\min\{m_1,m_2+\beta_2-\beta_1\}\quad \mbox{and }\nonumber\\
\tau_2(i_1)&=\max\{0,i_1+\beta_1-\beta_2\}.\nonumber 
\end{align}
In a similar way we find that $S_2$ satisfies  
\begin{align}
S_2&=\sum_{i_2=0}^{\mu_2}(-1)^{i_2}\cdot (i_2+\beta_2)\cdot \binom{m_2}{i_2}\cdot \sum_{i_1=\tau_1(i_2)}^{m_1}(-1)^{i_1}\cdot\binom{m_1}{i_1}, \;\mbox{where}\quad \label{eqn|S2} \\
\mu_2&=\min\{m_2,m_1+\beta_1-\beta_2-1\}\quad \mbox{and }\nonumber\\
\tau_1(i_2)&=\max\{0,i_2+\beta_2-\beta_1+1\}.\nonumber 
\end{align}
The summation $S_1$ will now be broken in two: The first, $S_{11}$ contains the terms that correspond to $i_1$ in the range $[0, \beta_2-\beta_1]$.  The summation $S_{12}$ contains the terms $i_1$ in the range $[\beta_2-\beta_1+1,\mu_1]$. We will prove that $S_{11}$ is $0$. The summation $S_{12}$ will have simple lower bound $\tau_2(i_1)$ that is always equal to $i_1+\beta_1-\beta_2$.
The number $S_{11}$ is obviously $0$ if $\beta_2<\beta_1$. If $\beta_2\geq\beta_1$ and $i_1\leq \beta_2-\beta_1$, then $i_1+\beta_1-\beta_2\leq 0$. This implies that $\tau_2(i_1)=0$. The number $S_{11}$ becomes 
\begin{eqnarray*}
S_{11}&=&\sum_{i_1=0}^{\beta_2-\beta_1}(-1)^{i_1}\cdot (i_1+\beta_1)\cdot \binom{m_1}{i_1}\cdot \sum_{i_2=0}^{m_2}(-1)^{i_2}\cdot\binom{m_2}{i_2}.
\end{eqnarray*}
The inner summation $\sum_{i_2=0}^{m_2}(-1)^{i_2}\cdot\binom{m_2}{i_2}$ is equal to $0$ because it is equal to $(1-1)^{m_2}$. Therefore, $S_{11}=0$ and $S_1=S_{12}$, i.e. 
\begin{align} 
S_1&=\sum_{i_1=\max\{0,\beta_2-\beta_1+1\}}^{\min\{m_1,m_2+\beta_2-\beta_1\}}(-1)^{i_1}\cdot (i_1+\beta_1)\cdot \binom{m_1}{i_1}\cdot \sum_{i_2=\tau_2(i_1)}^{m_2}(-1)^{i_2}\cdot\binom{m_2}{i_2}. \label{eqn|S12} 
\end{align}
The equality \eqref{eqn|ABAlternation} transforms \eqref{eqn|S12} into
\begin{align} 
S_1&=\sum_{i_1=\max\{0,\beta_2-\beta_1+1\}}^{\min\{m_1,m_2+\beta_2-\beta_1\}}(-1)^{i_1}\cdot (i_1+\beta_1)\cdot \binom{m_1}{i_1}\cdot  (-1)^{\tau_2(i_1)}\cdot \binom{m_2-1}{\tau_2(i_1)-1}\nonumber \\
&=(-1)^{\beta_1-\beta_2}\cdot \sum_{i_1=\max\{0,\beta_2-\beta_1+1\}}^{\min\{m_1,m_2+\beta_2-\beta_1\}}  (i_1+\beta_1)\cdot \binom{m_1}{i_1}\cdot  \binom{m_2-1}{i_1+\beta_1-\beta_2-1}.
 \label{eqn|S12S} 
\end{align}

We treat the summation $S_2$ in a similar way. As before, $S_{21}$ gathers the terms for which $i_2$ is in the range $[0, \beta_1-\beta_2-1]$. Of course, if $\beta_1-\beta_2-1<0$, then there is no summation $S_{21}$. The summation $S_{22}$ contains the terms for which $i_2$ is in the range $[\beta_1-\beta_2,\mu_2]$. It is straightforward to prove that $S_{21}$ is $0$. Therefore,  
\begin{align}S_2&=(-1)^{\beta_2-\beta_1+1}\cdot\sum_{i_2=\max\{0,\beta_1-\beta_2\}}^{\min\{m_2,m_1+\beta_1-\beta_2-1\}}  (i_2+\beta_2)\cdot 
\binom{m_2}{i_2}\cdot
 \binom{m_1-1}{i_2+\beta_2-\beta_1} . \label{eqn|S22}
\end{align}
The summations $S_1$ and $S_2$ have the opposite signs.  The summation $S_1$ has the term $(i_1+\beta_1)$ that makes it possible to split the summation into two simpler ones $T_{11}$ and $T_{12}$. In a similar way we split $S_2$ into $T_{21}$ and $T_{22}$. 
\begin{align}
S_1&=(-1)^{\beta_1+\beta_2}\cdot (T_{11}+T_{12}), \quad
S_2=-(-1)^{\beta_1+\beta_2}\cdot (T_{21}+T_{22}), \;\mbox{where}\quad \label{eqn|sumSplitting} \\
T_{11}&= \sum_{i_1=\max\{0,\beta_2-\beta_1+1\}}^{\min\{m_1,m_2+\beta_2-\beta_1\}} i_1\cdot \binom{m_1}{i_1}\cdot  \binom{m_2-1}{i_1+\beta_1-\beta_2-1},\label{eqn|T11B}\\
T_{12}&=\beta_1\cdot\sum_{i_1=\max\{0,\beta_2-\beta_1+1\}}^{\min\{m_1,m_2+\beta_2-\beta_1\}}   \binom{m_1}{i_1}\cdot  \binom{m_2-1}{i_1+\beta_1-\beta_2-1},\label{eqn|T12}\\
T_{21}&=\sum_{i_2=\max\{0,\beta_1-\beta_2\}}^{\min\{m_2,m_1+\beta_1-\beta_2-1\}}  i_2\cdot 
\binom{m_2}{i_2}\cdot
 \binom{m_1-1}{i_2+\beta_2-\beta_1}, \label{eqn|T21B}\\
 T_{22}&=\beta_2\cdot\sum_{i_2=\max\{0,\beta_1-\beta_2\}}^{\min\{m_2,m_1+\beta_1-\beta_2-1\}}  
\binom{m_2}{i_2}\cdot
 \binom{m_1-1}{i_2+\beta_2-\beta_1}.\label{eqn|T22}
\end{align}
The summations $T_{11}$ and $T_{21}$ can be simplified by first observing that the summations cannot ever have a term corresponding to the index $0$. Then for $i_1,i_2\neq 0$, we use $i_1\cdot \binom{m_1}{i_1}=m_1\cdot\binom{m_1-1}{i_1-1}$ and $i_2\cdot \binom{m_2}{i_2}=m_2\cdot\binom{m_2-1}{i_2-1}$. After using these identities, we can shift the indices $i_1$ and $i_2$ by $1$ and obtain 
\begin{align}
T_{11}&= m_1\cdot\sum_{i_1=\max\{0,\beta_2-\beta_1\}}^{\min\{m_1-1,m_2+\beta_2-\beta_1-1\}}  \binom{m_1-1}{i_1}\cdot  \binom{m_2-1}{i_1+\beta_1-\beta_2},\label{eqn|T11}\\
T_{21}&=m_2\cdot\sum_{i_2=\max\{0,\beta_1-\beta_2-1\}}^{\min\{m_2-1,m_1+\beta_1-\beta_2-2\}}  
\binom{m_2-1}{i_2}\cdot
 \binom{m_1-1}{i_2+1+\beta_2-\beta_1}. \label{eqn|T21}
\end{align}
In each of \eqref{eqn|T11}, \eqref{eqn|T12}, \eqref{eqn|T21}, and \eqref{eqn|T22}, we apply $\binom{n}{k}=\binom{n}{n-k}$ to the second binomial coefficient. The sums turn into 
\begin{align}
T_{11}&= m_1\cdot\sum_{i_1=\max\{0,\beta_2-\beta_1\}}^{\min\{m_1-1,m_2+\beta_2-\beta_1-1\}}  \binom{m_1-1}{i_1}\cdot  \binom{m_2-1}{
m_2-1-
i_1-\beta_1+\beta_2},\label{eqn|T11R}\\
T_{12}&=\beta_1\cdot\sum_{i_1=\max\{0,\beta_2-\beta_1+1\}}^{\min\{m_1,m_2+\beta_2-\beta_1\}}   \binom{m_1}{i_1}\cdot  \binom{m_2-1}{m_2-i_1-\beta_1+\beta_2},\label{eqn|T12R}\\
T_{21}&=m_2\cdot\sum_{i_2=\max\{0,\beta_1-\beta_2-1\}}^{\min\{m_2-1,m_1+\beta_1-\beta_2-2\}}  
\binom{m_2-1}{i_2} \nonumber\\
&\quad\quad\times 
 \binom{m_1-1}{m_1-1-i_2-1-\beta_2+\beta_1}, \label{eqn|T21R}\\
  T_{22}&=\beta_2\cdot\sum_{i_2=\max\{0,\beta_1-\beta_2\}}^{\min\{m_2,m_1+\beta_1-\beta_2-1\}}  
\binom{m_2}{i_2}\cdot
 \binom{m_1-1}{m_1-1-i_2-\beta_2+\beta_1}.\label{eqn|T22R}
\end{align}
We apply \eqref{eqn|CatsAndDogs} to each of \eqref{eqn|T11R}, \eqref{eqn|T12R}, \eqref{eqn|T21R}, and \eqref{eqn|T22R}.
\begin{align}
T_{11}&= m_1\cdot \binom{m_1+m_2-2}{m_2-1+\beta_2-\beta_1},\label{eqn|T11R1}\\
T_{12}&=\beta_1\cdot \binom{m_1+m_2-1}{m_2+\beta_2-\beta_1},\label{eqn|T12R1}  \\
T_{21}&=m_2\cdot \binom{m_1+m_2-2}{m_1+\beta_1-\beta_2-2}=m_2\cdot \binom{m_1+m_2-2}{m_2+\beta_2-\beta_1}, \label{eqn|T21R1}\\
T_{22}&=\beta_2\cdot \binom{m_1-1+m_2}{m_1-1+\beta_1-\beta_2}=\beta_2\cdot \binom{m_1+m_2-1}{m_2+\beta_2-\beta_1}.\label{eqn|T22R1}
\end{align}
From \eqref{eqn|T11R1} and  \eqref{eqn|T21R1} we obtain 
\begin{align}
T_{11}-T_{21}&=
\frac{(m_1+m_2-2)!\cdot \left((m_1+m_2)(\beta_2-\beta_1)+m_2\right)}{(m_2+\beta_2-\beta_1)!\cdot(m_1-1+\beta_1-\beta_2)!}\nonumber\\
&=\binom{m_1+m_2-1}{m_2+\beta_2-\beta_1}\cdot \frac{ \left((m_1+m_2)(\beta_2-\beta_1)+m_2\right)}{m_1+m_2-1}.
\label{eqn|T11mT21}
\end{align} 
We now use \eqref{eqn|S1PlusS2},  
\eqref{eqn|sumSplitting}, \eqref{eqn|T11mT21}, \eqref{eqn|T12R1}, and 
\eqref{eqn|T22R1} to obtain
\begin{align}
\frac{(-1)^{m_1+m_2}}{b-a}\partial_Sf(\omega)
&=(-1)^{\beta_1+\beta_2}\cdot\left(T_{11}+T_{12}-T_{21}-T_{22}\right).
\label{eqn|AfterSubstitution}
\end{align}
Substituting \eqref{eqn|T11mT21}, \eqref{eqn|T12R1}, and \eqref{eqn|T22R1} 
and simplifying yields
\begin{align}
\frac{(-1)^{m_1+m_2}}{b-a}\partial_Sf(\omega)
&=(-1)^{\beta_1+\beta_2}\cdot \binom{m_1+m_2-1}{m_2+\beta_2-\beta_1}\nonumber\\
&\times\left(\frac{\left((m_1+m_2)(\beta_2-\beta_1)+m_2\right)}{m_1+m_2-1}
+\beta_1-\beta_2\right) \nonumber \\
&=(-1)^{\beta_1+\beta_2}\cdot \binom{m_1+m_2-1}{m_2+\beta_2-\beta_1}
\cdot \frac{m_2+\beta_2-\beta_1}{m_1+m_2-1}. \label{eqn|Beforenm1chkm1}
\end{align}
The equation \eqref{eqn|Beforenm1chkm1} implies \eqref{eqn|dLanes} 
due to $\binom{n}{k}\cdot \frac{k}{n}=\binom{n-1}{k-1}$.
\end{proof} 

\end{appendix}

\end{document}